\documentclass[a4paper,10pt]{amsart}

\usepackage[T1]{fontenc}
\usepackage[latin1]{inputenc}
\usepackage{times}
\usepackage{amsmath}
\usepackage{amssymb}
\usepackage{amsthm}
\usepackage[mathcal]{euscript}
\usepackage{graphicx}
\usepackage{amscd}

\newcommand{\rem}[1]{}

\newcommand{\R}{\ensuremath{\mathbb{R}}}

\newcommand{\Q}{\ensuremath{\mathbb{Q}}}
\newcommand{\N}{\ensuremath{\mathbb{N}}}
\newcommand{\A}{\ensuremath{\mathbb{A}}}
\newcommand{\Fe}{\ensuremath{\mathcal{F}}}
\newcommand{\Ge}{\ensuremath{\mathcal{G}}}
\newcommand{\Ce}{\ensuremath{\mathcal{C}}}

\newcommand{\ind}{}

\newcommand{\Ordery}{\ensuremath{\textnormal{ord}_y}}
\newcommand{\Orderz}{\ensuremath{\textnormal{ord}_z}}
\newcommand{\ord}{order} 

\newcommand{\Degreey}{\ensuremath{\textnormal{deg}_y}}
\newcommand{\Degreez}{\ensuremath{\textnormal{deg}_z}}
\newcommand{\degreey}{degree}

\newcommand{\Adj}{\ensuremath{\textnormal{adj}}}
\newcommand{\Bonus}{\ensuremath{\textnormal{bonus}}}
\newcommand{\bonus}{bonus }
\newcommand{\Defect}{\ensuremath{\textnormal{defect}}}
\newcommand{\defect}{defect }
\newcommand{\const}{\ensuremath{\delta}}

\newcommand{\Indent}{\ensuremath{\textnormal{indent}}}
\newcommand{\indt}{indent }
\newcommand{\Updent}{\ensuremath{\textnormal{updent}}}
\newcommand{\updent}{updent }
\newcommand{\Dent}{\ensuremath{\textnormal{dent}}}
\newcommand{\dent}{dent }

\newcommand{\Slope}{\ensuremath{\textnormal{slope}}}
\newcommand{\slope}{slope}

\newcommand{\Parity}{\ensuremath{\textnormal{par}}}

\newcommand{\Dorder}{\ensuremath{\textnormal{shade}}}    
\newcommand{\dorder}{shade}    
\newcommand{\Revdorder}{\ensuremath{\textnormal{complicacy}}}    
\newcommand{\revdorder}{complicacy }   

\newcommand{\Height}{\ensuremath{\textnormal{height}}}
\newcommand{\height}{\ensuremath{\textnormal{height}}}
\newcommand{\Revheight}{\ensuremath{\textnormal{intricacy}}}  
\newcommand{\revheight}{intricacy}   

\newcommand{\I}{\ensuremath{\mathcal{I}}}
\newcommand{\ttop}{\ensuremath{\textnormal{top}}}
\newcommand{\weak}{{\ensuremath{\curlyvee}}}  
\newcommand{\JJ}{\ensuremath{\mathcal{J}}}

\newcommand{\Width}{\ensuremath{\textnormal{width}}}
\newcommand{\Order}{\ensuremath{\textnormal{ord}}}

\theoremstyle{plain}
\newtheorem{Theo}{Theorem}
\newtheorem{Prop}{Proposition}
\newtheorem{Lem}{Lemma}
\newtheorem{Cor}{Corollary}

\theoremstyle{remark}
\newtheorem{Bem}{Remark}

\newtheorem{abbildung}{Figure}

\input prepictex
\input pictex
\input postpictex

\title[Embedded Resolution of Surface Singularities in Positive Characteristic]{\large{Alternative Invariants for the Embedded Resolution of Purely Inseparable Surface Singularities}}
\author{H. HAUSER and D. WAGNER}
\thanks{Partially supported by projects 
P18992 and P21461 of the Austrian Science Fund FWF, by the Austrian-Spanish
cooperation program Acciones Integradas, by the University of Innsbruck and by a grant from L'Or\'eal Austria, the Austrian Commission for UNESCO and the Austrian Academy of Sciences.}
\date{\today}

\begin{document}

\maketitle


\begin{abstract}
\noindent We propose two local invariants for the inductive proof of the embedded resolution of purely inseparable surface singularities of order equal to the characteristic. The invariants are built on an detailed analysis of the so called kangaroo phenomenon in positive characteristic. They thus measure accurately the algebraic complexity of an equation defining a surface singularity in characteristic $p$. As the invariants are shown to drop after each blowup, induction applies. \end{abstract}


{ \small
\tableofcontents }


\section{Introduction} \label{aim}
$ $ \\
\noindent  In this paper, two alternative invariants for the embedded resolution of two-dimensional hypersurface singularities in arbitrary characteristic are constructed. The first invariant is built on the now classical invariant from characteristic zero, consisting of a string of integers given by the local order of the defining equation and of the orders of the subsequent coefficient ideals (after having removed the exceptional factors). As hypersurfaces of maximal contact need not exist in positive characteristic, these orders have to be defined in a different way to make them intrinsic. The correct choice is the {\it maximum} of the order of the coefficient ideal over {\it all} choices of local regular hypersurfaces. The orders are thus well defined, i.e., independent of any choices.

\ind By examples of Moh it is known that this invariant may increase under blowup with respect to the lexicographical order \cite{MR935710, MR1395176}. Actually, its second component, the order of the first coefficient ideal, may increase at points where the first component has remained constant. The increase occurs at so called {\it kangaroo points} (in Hauser's terminology; they are called {\it metastatic points} by Hironaka). Moh was able to bound the possible increase from above, and Hauser gave a complete classification of kangaroo points  \cite{Hausera, HauserBull}.

\ind Relying on these results, we show in the present paper (for purely inseparable two-dimen\-sional hypersurfaces of order equal to the characteristic) that the sporadic increase of the invariant is dominated by larger decreases before or after the critical blowup. It thus decreases in the long run. Actually, to smooth the argument and to avoid considering packages of blowups, we subtract from the second component of the invariant in very specific situations a {\it bonus} 
(a real number taking values $0$, $\varepsilon$, $\delta$ or $1+\delta$ with $0<\varepsilon<\delta<1$). 
 
\ind This bonus is modeled so that the modified invariant decreases under {\it every} blowup (see Theorem \ref{Theo1}). It thus interpolates the ``graph'' of the original characteristic zero invariant by a monotonously decreasing function (see figure \ref{diagram_modi}).

 	\begin{figure}[h]
	\begin{center}
	\includegraphics[scale=0.45]{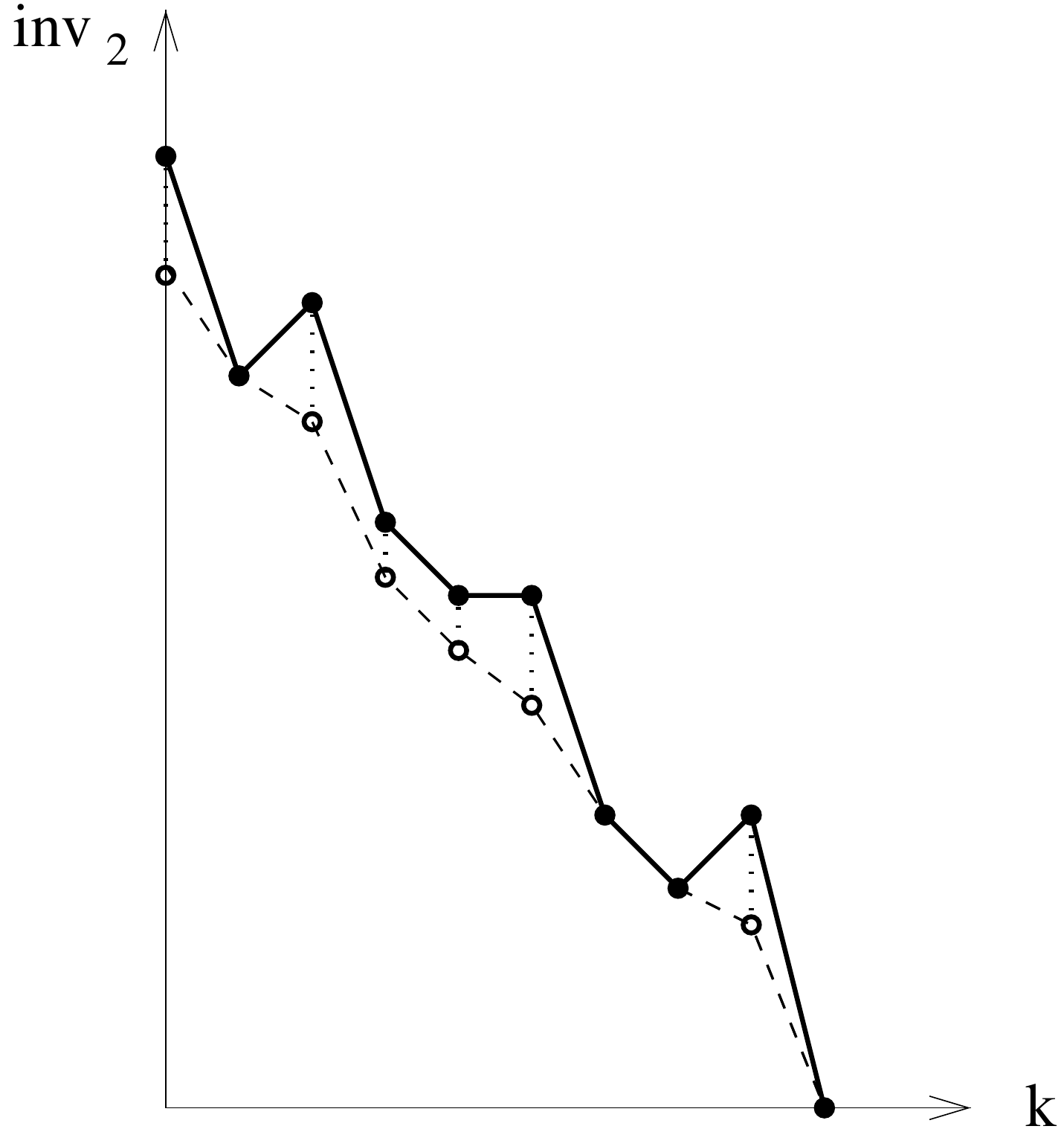}   
	\\ 
	[1ex]  \begin{minipage}{12cm} 
       \begin{abbildung}\label{diagram_modi} \begin{center}
       Modification of the classical invariant (solid line) by the bonus (dashed);  vertically the 2nd component of the invariant, horizontally the number of blowups.
        \\ \end{center} \end{abbildung} 
      \end{minipage}
	\end{center}
	\end{figure}

\noindent Our second invariant is built on a different measure, the {\it height}. This is a natural number which counts in an asymmetric way the distance of a hypersurface singularity from being a normal crossings divisor. The symmetry is broken by the consideration of local flags which accompany the resolution process. They reduce the necessary coordinate changes to a ``Borel'' subgroup of the local formal automorphism group of the ambient scheme: the changes are {\it triangular} in a precise sense. This, in turn, allows us to define the height as a minimum over all coordinate choices subordinate to the flag. Moreover, the local blowups given by choosing an arbitrary point in the exceptional divisor can be made monomial after applying at the base point below a suitable linear triangular coordinate change belonging to the subgroup. Combining these techniques one obtains an explicit control on the behavior of the height under blowup.  

\ind After completion of the present paper we became aware through Cutkosky's recent preprint \cite{Cutkosky-on-Abhyankar} that an invariant similar to the height had already been considered by Abhyankar in a series of papers from the sixties, see e.g. \cite{Abhyankar-nonsplitting}. His more valuation theoretic approach is very complicated. Cutkosky simplifies considerably Abhyankar's constructions and thus achieves a lucid exposition of the induction argument. For a comparison of the invariants, see in particular definition 7.3 of \cite{Cutkosky-on-Abhyankar}.

\ind Experimentation shows that the height may also increase under blowup, as was the case for the order of the coefficient ideal. But Moh's bound applies again. In fact, the bonus which has to be subtracted to make the resulting invariant always drop is now much easier to define than before. It is $0$, $\varepsilon$ or $1+\delta$ according to the situation, with $0<\varepsilon<\delta <1$. As a consequence we can show quite directly that the vector of (modified) heights (of the subsequent coefficient ideals) drops  lexicographically under blowup 
(again in the case of purely inseparable two-dimensional hypersurface singularities of order equal to the characteristic). 

\ind Both types of invariants as well as the respective definitions of the bonus provide a quite concise approach to the resolution of surface singularities. They thus form a substitute for Hironaka's invariant from the Bowdoin lectures \cite{Hironaka}, which is central in the recent works in positive characteristic of Cossart-Jannsen-Saito \cite{CJS} on the embedded resolution of surfaces of arbitrary codimension and of Cutkosky \cite{Cutkosky_3-Folds} and Cossart-Piltant \cite{Cos1,Cos2} on the non-embedded resolution of three-dimensional varieties. All these proofs use Hironaka's invariant for surfaces. See also \cite{FaberHauser} for many concrete examples. A different approach to the resolution of surface singularities has been proposed by Benito-Villamayor, which was then simplified by Kawanoue-Matsuki \cite{Benito_Villamayor_Surfaces, Kawanoue_Matsuki_Surfaces}.

\ind It is appropriate to compare the invariants proposed in this article with Hironaka's. All three can be defined through the Newton polyhedron of the singularity. They are made intrinsic by astute choices of local coordinates, and thus serve as genuine measures of the complexity of the singularity, not depending on any casual instance or choice. \\

\noindent {\it Advantages of the invariants}: (1) They are very natural and easy to handle. \rem{This ensures their upper-semicontinuity required for the induction argument. } (2) Their construction is systematic. This permits us to investigate possible extensions to higher dimensions (though there are then various options of how to design them).  (3) They do not increase even if the center was chosen too small (i.e., a point instead of a curve). This is not the case with Hironaka's invariant which requires to blow up in a center of maximal possible dimension. In contrast, for our invariants, the centers of blowup will always be a collection of isolated points, except if the first coefficient ideal is a monomial (the $\nu$-quasi-ordinary case; this is a purely combinatorial situation). (4) The symmetry break in the definition of the second invariant may result fertile in the future. The proofs show that this is an efficient way to control blowups. It is built on the asymmetric decomposition of projective space (typically, the exceptional divisor of a point blowup) by affine spaces of decreasing dimensions. We thus partition the exceptional divisor by locally closed subsets instead of covering it by open affine subsets. The flags take into account this decomposition. (5) The bonus is based on a detailed analysis of the kangaroo phenomenon. The increase of the not yet modified invariant \`a la Moh under blowup can be shown to come along with a complementary improvement of the Newton polyhedron: It approaches a coordinate axis. Exploiting this incidence, first observed by Dominik Zeillinger in his thesis \cite{Zeillinger}, the definition of the bonus comes quite automatically. (6) The proofs that the (modified) invariant drops are completely straightforward and thus -- at least in principle -- extendable to higher dimension. \\

\noindent {\it Drawbacks of the invariants}: 
(7) The maximal order of the coefficient ideal over all coordinate changes, called here the {\it \dorder} of the singularity (which coincides  in the purely inseparable case with the {\it residual order} of Hironaka), is not upper-semicontinuous when considering non-closed points. Hironaka calls this phenomenon {\it generic going up}. It causes technical complications in higher dimensions.  (8) The introduction of the bonus is not completely satisfactory. It ensures that the modi\-fied invariant drops after each blowup, but its definition could be more conceptual, avoiding case distinctions. 
 (9) The extension of the results and techniques to the embedded reso\-lution of threefolds -- this is known to be the critical case for positive characteristic -- is not obvious. There seem to appear additional complications which are not entirely understood yet. 

\ind Our exposition concentrates on purely inseparable hypersurface singularities of order $p$ (which represent the first significant case.)  If the order is a larger $p$-th power $p^k$, $k\geq 2$, the reasoning becomes more complicated. For instance, the bonus (defined in  section 4) has to be modified from $1+\delta$ to $p^{k-1} +\delta$. A similar argument as developed in this paper shows that a sequence of point blowups reduces to the case where the width (see section 4) has become $\leq p^{k-1}$. This case, however, seems to be much more intricate than the case $k=1$ where the width is just $1$. It is the subject of ongoing research.\\

 

\noindent {\bf Acknowledgements. } The authors wish to express their thanks to Dominik Zeillinger for sharing generously his ideas. Thanks to Tobias Beck, Clemens Bruschek, Santiago Encinas, Daniel Panazzolo, Georg Regensburger, Dale Cutkosky, Josef Schicho and Hiraku Kawanoue for several helpful comments and many stimulating conversations.
\\


\section{Context} \label{context}
$ $ \\
\noindent Hironaka's proof of resolution of singularities in characteristic zero in \cite{MR0199184} is built on induction on the dimension of the ambient space. This descent in dimension persists as the key argument also in the later simplifications of Hironaka's proof by Villamayor, Bierstone-Milman, Encinas-Hauser, Bravo-Encinas-Villamayor, W\l odarczyk, Koll\'ar \cite{MR985852, MR1198092, MR1440306, MR1949115, MR2174912, MR2163383, MR2289519}: To an ideal sheaf $\I$ in an $n$-dimensional, smooth ambient scheme $W$ one associates locally at each point $a$ of $W$ 
a smooth hypersurface $V$ of $W$ through $a$ and an ideal sheaf $\JJ$ in $V$, the coefficient ideal of $\I$ in $V$ at $a$, which translates the resolution problem for $\I$ in $W$ at $a$ into a resolution problem of $\JJ$ in $V$. Once $\JJ$ is resolved -- this can be assumed to be feasable by induction on the dimenson $n$ -- there is a relatively simple combinatorial procedure to also resolve $\I$. 

Let us recall here that there exist various proofs for (embedded, respectively non-embedded) resolution of surfaces in arbitrary characteristic. Abhyankar's thesis \cite{Ab8} from 1956, Lipman's proof in \cite{Lipman} via pseudo-rational singularities for arbitrary 2-dimensional excellent schemes (but dispensing of embeddedness), and Hironaka's proof from his Bowdoin lectures \cite{Hironaka}, where an invariant is constructed from the Newton polyhedron of a hypersurface. This proof is used in the recent work of Cossart-Jannsen-Saito \cite{CJS} on embedded resolution of two-dimensional schemes, Cutkosky's compact writeup \cite{Cutkosky-on-Abhyankar} of Abhyankar's scattered proof of non-embedded resolution for threefolds in positive characteristic $> 5$ (hypersurface case), and the papers \cite{Cos1, Cos2} of Cossart and Piltant, where the result is established with considerably more effort for arbitrary reduced three-dimensional schemes defined over a field of positive characteristic which is differentially finite over a perfect subfield.   

Moreover lately there have been new developments in the area of resolution of singularities of algebraic varieties of any dimension over fields of positive characteristic. 
For instance, several promising new approaches and programs
have been presented during the conference ``On the Resolution of Singularities'' at RIMS Kyoto in December 2008: In \cite{MR1996845, HironakaTrieste, HironakaClay} Hironaka studies differential operators in arbitrary characteristic in order to construct generalizations of hypersurfaces of maximal contact. The main 
difficulty is thus reduced to the purely inseparable case and kangaroo/metastatic points. 
Hironaka then asserts that this type of singularities can 
be resolved directly \cite{HironakaRIMS, HironakaTordesillas}. There is no written proof of this available yet.
Further Kawanoue and Matsuki have published a program for arbitrary dimension and characteristic \cite{Hiraku1, Kawanoue_Matsuki}. Again differential operators are used to define a suitable resolution invariant. The termination of the resulting algorithm seems not to be ensured yet. Additionally there is a novel approach to resolution by Villamayor and his collaborators Benito, Bravo and Encinas \cite{EV,Bravo_Villamayor_10,Benito}.   
It is based on projections instead of restrictions for the descent in dimension. A substitute for coefficient ideals is constructed 
via Rees algebras and differential operators, called elimination algebras. It provides a new 
resolution invariant for characteristic $p$ (which coincides with the classical one in characteristic zero).
This allows one to reduce to a so 
called monomial case (which, however, seems to be still unsolved, and could be much more 
involved than the classical monomial case). 

A more axiomatic approach to resolution has been proposed by Hauser and Schicho \cite{HauSchi}: The various specific constructions of the classical proof in characteristic zero are replaced by their key properties. These in turn suffice to give a purely combinatorial description of the entire resolution argument in form of a game (a viewpoint which orignally goes back to Hironaka). To get a complete proof of resolution then one only has to show, and this is done by elementary algebra, that objects with the required properties do exist.

In the course of Hironaka's reasoning of resolution of singularities in characteristic zero it is crucial that the local descent in dimension commutes with blowups in admissible centers (= smooth centers contained in $\ttop(\I)$) at all points of the exceptional divisor $Y'$ where the local order of $\I$ has remained constant. More explicitly, this signifies that the coefficient ideal of the weak transform $\I^\weak$ of $\I$ at a point $a'$ of $Y'$ where the order of $\I$ has remained constant equals the (controlled) transform of the coefficient ideal of $\I$ at $a$ (for the involved notions of coefficient ideal, weak and controlled transforms, see \cite{MR1949115}).

The commutativity of the local descent to coefficient ideals with blowups is essential for proving that -- always in characteristic zero --  the order of the coefficient ideal $\JJ$ of $\I$ does not increase at points where the order of $\I$ has remained constant. (It is easy to see, using that the center is contained in $\ttop(\I)$, that the order of $\I$ itself cannot increase.) Therefore the pair $(\textnormal{ord}_a(\I),\textnormal{ord}_a(\JJ))$ does not increase under blowup when considered with respect to the lexicographic order.

The clue for this to work is the existence of hypersurfaces of maximal contact in characteristic zero. They are special choices of hypersurfaces $V$ containing locally $\ttop(\I)$ at $a$ and ensuring that the strict transform $V^{st}$ of $V$ contains again the top locus of the weak transform $\I^\weak$ of $\I$, provided the maximum value of the local orders of $\I$ has remained constant. Moreover, it is required that this property persists for $\I^\weak$ and $V^{st}$ under any further admissible blowup. In particular, the various transforms of $V$ contain all equiconstant points, i.e., points of the subsequent exceptional loci where the local order of the transforms has remained constant (at the other points, induction on the order applies).

This argument fails in positive characteristic. There are ideals in characteristic $p>0$ (first given by Narasimhan in \cite{MR684627} and \cite{MR715853}, then also studied by Mulay \cite{MR715854}), whose top locus is not locally contained in any smooth hypersurface. Consequently, when just taking any smooth hypersurface through the point $a$, its transforms under blowups eventually lose the equiconstant points of $\I$ (see \cite{Hausera} for the reason for this and a selection of examples). Hence the induction on the dimension breaks down in a first instance, because the descent in dimension does no longer commute with blowups in the above way.

In an attempt to overcome this flaw, one could choose after each blowup locally at equiconstant points $a'$ of the exceptional locus $Y'$ a new local hypersurface $V'$ (instead of the transform $V^{st}$ of $V$) and try to compare the resulting coefficient ideal with the one below in $V$. In trying to do this, one has to choose carefully the hypersurfaces $V$ and $V'$. The first should have transform $V^{st}$ containing all equiconstant points $a'$ in $Y'$ (for reasons not apparent at the moment), so that only a {\it local} automorphism at $a'$ is necessary to obtain $V'$ from $V^{st}$. Moreover, $V'$ should have the same property as $V$ -- but again only for the next blowup, not for all subsequent ones.

Additionally, a second condition is imposed on $V$. It is related to the construction of the resolution invariant. Usually, this invariant is a vector whose entries are the local orders of certain ideals: The first component is the order of $\I$ at $a$, the second the order of the coefficient ideal $\JJ$ of $\I$ at $a$ in $V$ (after having factored from it possible exceptional components). But this second order may depend on the choice of $V$, and we are better led to choose only such $V$ for which the order of the coefficient ideal takes an intrinsic value. 

In characteristic zero, another coincidence occurs. Hypersurfaces of maximal contact maximize the order of the coefficient ideal over all choices of local, smooth hypersurfaces. Thus, this order is intrinsic. In \cite{MR2163383}, W\l odarczyk introduced a version of coefficient ideal whose analytic isomorphism class does not depend on $V$, so that its local order is automatically intrinsic. The maximality leads naturally to the notion of {\it weak maximal contact}, which was introduced in \cite{MR1949115}: The local, smooth hypersurface $V$ through $a$ has weak maximal contact with $\I$ if the order of the coefficient ideal $\JJ$ of $\I$ in $V$ is maximized over all smooth local hypersurfaces. This notion depends, of course, on the selected definition of coefficient ideal.

Maximality of orders can be traced back in many papers, and was especially for Abhyankar a decisive requirement \cite{MR713043}. He achieved it in characteristic zero by so called Tschirnhaus transformations, an algebraic construction of local coordinate changes yielding hypersurfaces slightly stronger than hypersurfaces of maximal contact (the resulting hypersurfaces are called {\it osculating} in \cite{MR1949115}). \\


\section{Results} \label{results}
$ $ \\
\noindent The present paper originates from the observations indicated in section \ref{context}. It exhibits, still for surfaces, but with the perspective of application to higher dimensional schemes, a characteristic free approach to hypersurfaces of weak maximal contact and their related coefficient ideals. It was observed by Moh in \cite{MR935710} and \cite{MR1395176} that the order of the coefficient ideal of an ideal sheaf in a hypersurface of weak maximal contact may indeed increase in characteristic $p>0$ -- this was already known to Abhyankar \cite{Abhyankar-nonsplitting, Cutkosky-on-Abhyankar} -- and he was able to bound the increase. And in fact, the increase is small. If $\I$ is a principal ideal of order $p$ (the characteristic) at a given point, the increase of the order of the coefficient ideal is at most $1$ (always considered at equiconstant points of $\I$ in $Y'$, the only points of interest). This is not too bad, but, conversely, sufficient to destroy any kind of naive induction.

In the present paper we investigate this increase closer in the case of surfaces. It is known from Hauser's work that an increase can occur only rarely, and that the situations where an increase happens are very special and can be completely characterized \cite{Hausera}. However, it cannot be excluded that the increase repeats an infinite number of times. This would not rule out the existence of resolution in positive characteristic, but it would show that the characteristic zero resolution invariant formed by the orders of the successive coefficient ideals cannot be used directly in characteristic $p$. 

The point of the present paper is that, at least for surfaces, the same resolution invariant as in characteristic zero can be used also in characteristic $p$. It suffices to modify it slightly in some very specific circumstances to make it work again. The trick lies in subtracting occasionally a {\it \bonus} from the invariant. This is a correction term (taking values $1+\delta$, $\delta$, $\varepsilon$ or $0$ for  once chosen constants $0< \varepsilon <\delta <1$) which makes the modified invariant drop lexicographically after {\it each} blowup  --with a few exceptions, so called quasi-monomials, where a direct resolution of the surface can be given (a quasi-monomial occurs if the coefficient ideal is a product of the monomial exceptional factor with a polynomial of order $1$ at the given point, see section 4 for the precise definition).

The classical resolution invariant -- consisting of orders of successive coefficient ideals -- and its modification will be treated in section \ref{invariant2} of this paper. 
Instead we will define and work primarily with a new resolution invariant which was constructed in the thesis \cite{Zeillinger} of Zeillinger. As in the case of the classical characteristic zero resolution invariant, its 
components are related to the successive coefficient ideals. But instead of measuring the respective 
orders, we will associate to each of these ideals a certain ``height''. 
It measures in an asymmetric manner the distance of a hypersurface singularity from being a normal crossings divisor. We prefer this new resolution invariant because its correction term is easier to define and the induction argument becomes simpler. 

The critical case is the purely inseparable equation $$G=x^p+F(y,z)$$ with $\textnormal{ord}(G)=p$. 
The present paper, therefore,
concentrates on this situation. This smoothes the exposition, and avoids technical complications which occur if one wants to extend the argument
to arbitrary equations of surfaces (one would have to work with
coefficient ideals as defined in \cite{MR1949115, Cut1}).
Coefficient ideals correspond geometrically to the projection of the Newton polyhedron of $G$
from the point $x^{\textnormal{\small{ord}}(G)}$ onto the $yz$-coordinate 
plane (for more details we refer to \cite{MR935710, Hausera} and remark \ref{hironaka} in section \ref{Kap1}) and yields 
a resolution problem which has exactly the same features as the purely inseparable equation. 

The surfaces we are considering are embedded in a smooth three-dimensional algebraic variety over an algebraically closed field $K$. In general, such a variety does not admit a covering by open subsets isomorphic to open subsets of $\A_K^3$. To simplify the situation we pass to \'etale neighborhoods and thus work in the completion of the local rings. This makes the construction of invariants easier and allows us to restrict to the case that the completion of the local ring at a point is the quotient of a formal power series ring in three variables modulo a principal ideal. For simplicity of notation we will assume that this ideal is generated by a polynomial, i.e., that the surface is locally embedded in $\A_K^3$. The constructions in the general case are similar.

Therefore we will restrict to the case that $F$ and $G$ from above are elements of a polynomial ring $R$ over an algebraically closed field of positive characteristic. 
Coordinate changes of the form $x \rightarrow x+v(y,z)$ can be used to
eliminate $p$-th powers from the polynomial $F$ without
changing, up to isomorphism, the geometry of the algebraic variety
defined by $G$. Therefore it is natural to work in the quotient $Q=R/R^p$ of $R$ by the subring $R^p$ of $p$-th powers. Especially, the problem of the resolution of $G$ can be transferred to the problem of the monomialization of $F$ modulo $R^p$. It appears to be surprisingly difficult to extract substantial information on the complexity of the singularities of $G$ from the knowledge of $F$ up to $p$-th powers. Of course, any reasonable measure of complexity should not increase under blowup in smooth centers.

The invariant constructed by Hironaka \cite{Hironaka}
is built on coordinate independent
data extracted from the Newton polyhedron of the
defining equation in local coordinates. It has the drawback that its improvement under blowup
relies on the choice of an admissible center of {\it maximal
possible dimension}. Said differently, when a smooth curve can be
chosen as center (because the surface has constant order along the curve), it has to be chosen, otherwise the invariant may go up under blowup. It is precisely this restriction which makes it very hard, if not impossible, to generalize the invariant and the
induction argument of Hironaka to threefolds. 

\ind Our resolution invariants will also be constructed from the Newton polyhedron of $G$ in a coordinate independent manner. The first is primarily based on the measure ``\height'', which reflects in an asymmetric way the distance of the Newton polygon of $F$ from being a quadrant. The second builds on the characteristic zero invariant. In very specific situations -- according to special positions of the Newton polygon in the positive quadrant -- these invariants are adjusted by subtracting a ``bonus''. 

\ind We shall give a precise formulation and a systematic proof of the
following statement  (cf.  Theorem \ref{Theo11} in section \ref{Kap2}, and section \ref{Kap3}). The assumption of pure inseparability is not crucial.

\noindent 
\begin{Theo} \label{Theo1} Let $X$ be a singular surface in $\A^3$, defined over an algebraically closed field of characteristic $p>0$ by a purely inseparable equation of the form $$G(x,y,z)=x^p+F(y,z)$$ where $F$ is a polynomial of order $\geq p$ at $0$. Let $\tau: \widetilde{\A^3} \rightarrow \A^3$ be the blowup of $\A^3$ with center the origin, and let $\pi: \widetilde{\A^2} \rightarrow \A^2$ be the induced blowup of $\A^2=0 \times \A^2$ with exceptional divisor $E$.
Let $f$ be the residue class of $F$ modulo $p$-th powers and assume that $f$ is not a quasi-monomial. \smallskip

\noindent (i) There exists a local invariant $i_a(f)$ such that for any closed point $a'$ in $E$ one has $$i_{a'}(f') < i_a(f),$$ where $f'$ denotes the equivalence class of the transform $F'$ of $F$ modulo $p$-th powers.  \smallskip

\noindent (ii) Finitely many point blowups transform $f$ in any point of the exceptional divisor into a monomial or make the order of $G$ drop below $p$. 
\end{Theo} 

\ind 
The invariant $ i_a(f)$ is defined in section \ref{Kap1}. It can be shown that the set of closed points $a \in \A^2_K$ in which $f \in Q=R/R^p$ is not monomial consists of at most finitely many points.
Once $f$ is monomial, there exists a simple combinatorial method to decrease the order of $G$ by finitely many further point and curve blowups. 
Note that in contrast to Hironaka's invariant, which requires to choose in every step of the resolution algorithm a center of maximal possible dimension, we always blow up in a point until $f$ is monomial. Only in this situation 
curve blowups are possibly needed in order to lower the order of $G$.
Hence we achieve a proof of the following result:

\begin{Cor} \label{Theo2} Finitely many blowups of points and smooth curves will decrease the order of any purely inseparable singular two-dimensional hypersurface whose maximum of local orders is less or equal to the characteristic of the ground field. \end{Cor}


\ind Singularities of an arbitrary surface $X=V(G)$ in $\A^3_K$ with $\textnormal{ord}(G)<p$ can be resolved using the usual resolution algorithm from characteristic zero. Therefore, Theorem \ref{Theo1} and Corollary \ref{Theo2} imply 
the following statement: 

\begin{Cor} \label{corol2}
Finitely many suitable blowups of points and smooth curves yield an embedded resolution of a purely inseparable two-dimensional hypersurface $X$ whose maximum of local orders is less or equal to the characteristic of the ground field (i.e., the strict transform is smooth and the total transform is a normal crossings divisor.) \\
\end{Cor}


\section{The resolution invariant} \label{Kap1}
$ $ \\
\noindent In the last section we already indicated why resolution of purely inseparable surfaces $G=x^p+F(y,z)$ with $\textnormal{ord}_0(G)=p$ boils down to the monomialization of $F$ modulo $R^p$, where $R$ denotes the coordinate ring of the affine plane $\A^2_K$ and $R^p$ its subring of $p$-th powers.
Therefore we will in the sequel restrict to the study of polynomials $F(y,z)$ modulo $p$-th powers.   

Denote by $R_a$ the localization of $R$ at a closed point $a$ of $\A^2$ and $\widehat R_a$ its completion with respect to the maximal ideal. A regular parameter system $(y,z)$ of $\widehat R_a$ will be called {\it a system of local coordinates} of $R$ at $a$. Any choice of local coordinates $(y,z)$ induces an isomorphism of $\widehat R_a$ with the formal power series ring $K[[y,z]]$ corresponding to the Taylor expansion of elements of $R$ at $a$ with
respect to $y$ and $z$. Therefore, for any residue class $f \in R/R^p$, there is a unique expansion $F=\sum_{\alpha\beta} c_{\alpha\beta}y^{\alpha} z^{\beta}$ of $f$ in
$K[[y,z]]$ with $(\alpha,\beta)\in \N^2\setminus p\cdot \N^2$. This
corresponds to considering $\N^2$ with ``holes'' at the points of
$p\cdot \N^2$. We shall always distinguish carefully between elements $f$ in $R/R^p$ and their representatives $F$ as expansions $F(y,z)$ in $K[[y,z]]$ without any $p$-th power monomials. The dependence of $F$ on the coordinates $y$ and $z$ is always tacitly assumed without extra notation. The passage to the completion is necessary to dispose of a flexible notion of isomorphism.

A {\it local flag} $\Fe$ in $\A^2$ at $a$ is a regular element $h$ of
$\widehat R_a$ (cf. \cite{MR2063657}). Coordinates $(y,z)$ are called {\it subordinate to
the flag} if $z$ and $h$ generate the same ideal in $K[[y,z]]$. We
denote by $\Ce=\Ce_\Fe$ the set of subordinate local coordinates.
{\it Subordinate coordinate changes} are automorphisms of $K[[y,z]]$ which preserve subordinate coordinates. They are of the form $(y,z) \rightarrow (y+ v(y,z), z\cdot u(y,z))$ with series $v(y,z), u(y,z) \in K[[y,z]]$ satisfying
$\partial_yv(y,0)\neq -1$ and $u(0,0)\neq 0$. 

\ind We will first define measures which capture the ``distance'' of the expansion $F(y,z)$ of $f \in R/R^p$ at $a$ with respect to  fixed subordinate coordinates $(y,z)$ from being a monomial -- up to multiplication by units in $K[[y,z]]$. Afterwards these measures will be made coordinate independent in order to establish a local resolution invariant $i_a(f)$ for residue classes $f \in R/R^p$. 
\\


 \begin{figure}[h]
 \begin{center}
 \begin{minipage}{6cm}
 \beginpicture

 \setcoordinatesystem units <7pt,7pt> point at 0 0

 \setplotarea x from -1 to 14, y from -1 to 14
 \axis left shiftedto x=0 ticks numbered from 5 to 14 by 5
                                short unlabeled from 1 to 14 by 1 /
 \axis bottom shiftedto y=0 ticks numbered from 5 to 14 by 5
                                short unlabeled from 1 to 14 by 1 /

 \put {\circle*{5}} [Bl] at  1  8
 \put {\circle*{5}} [Bl] at  3  5
 \put {\circle*{5}} [Bl] at  6  2
 \put {\circle*{5}} [Bl] at  11 1
 \put {\circle*{5}} [Bl] at  6  11
 \put {\circle*{5}} [Bl] at  9  5
 \put {\circle*{5}} [Bl] at  2 9

 \plot 1 14  1 8   3 5  6 2   11 1  14 1  /   

 \vshade  1 8 14 <,z,,> 3 5 14  <,z,,> 6 2 14 <,z,,> 11 1 14 <,z,,> 14 1 14 /

 \multiput {\multiput {\circle{5}} [Bl] at 0  0 *2 5 0 /} [Bl] at 0 0 *2 0 5 /
 \put {$z$} at 14 -0.8
 \put {$y$} at -0.8 14
 \endpicture

 \end{minipage}
 \\[2ex]
 \begin{minipage}{11.5cm}
 \begin{abbildung}\label{Newton Polygon} \begin{center} Newton polygon  of an element $F \in K[[y,z]]/K[[y,z]]^p$ with $p=5$. The monomials of $K[[y,z]]^p$ are indicated by ``holes'' $\circ$ at the points of $p \cdot \N^2$. \end{center} \end{abbildung}
 \end{minipage}
 \end{center}
 \end{figure}

\noindent The Newton polygon $N=N(F)$ of an element $F\in K[[y,z]]$ is the
positive convex hull $\textnormal{conv}(\textnormal{supp}(F)+\R_+^2)$ of the
support $\textnormal{supp}(F) =\{(\alpha,\beta)\in \N^2 \setminus p \cdot \N^2; 
c_{\alpha\beta}\neq 0\}$ of $F$. Newton polygons will be depicted in
the positive quadrant of the real plane $\R^2$, the $y$-axis being chosen
vertically, the $z$-axis to the right (see figure \ref{Newton
Polygon}). 

\ind Let $A \subset \N^2 \setminus p \cdot \N^2$ be the set of vertices of the Newton
polygon $N$ of $F$, i.e., the minimal set such that
$N=\textnormal{conv}(A+\R^2_+)$. 
The {\it order} of $F$ at $0$ is defined as $$\textnormal{ord}(F)=\min_{(\alpha,\beta) \in A} \ \alpha+\beta,$$ i.e., as the order of $F$ as a power series in $K[[y,z]]$. Note that $\textnormal{ord}(F)$ takes the same value for all coordinates $(y,z) \in \Ce$, it thus depends only on $f$ and $a$. It will be called the {\it order}  of $f \in R/R^p$ at $a$, denoted by $\textnormal{ord}_a(f)$. The {\it initial form} $f_d$ of $f$ at $a$ is the residue class of $f$ modulo $\mathfrak{m}^{d+1}$, where $d=\textnormal{ord}_a(f)$ and  $\mathfrak{m}$ denotes the maximal ideal of $R$ at $a$. Given $y$ and $z$ it is induced by the homogeneous form $F_d$ of lowest degree $d$ of the expansion $F$ of $f$, say $F=F_d+F_{d+1}+\ldots$, with $F_d\neq 0$.    Furthermore denote by
$$\Ordery(F) = \min_{(\alpha,\beta) \in A} \ \alpha$$
the {\it \ord} of $F$ with respect to $y$ (symmetric definition for $\Orderz(F)$). This is just the order of vanishing of $F$ along the curve $y=0$. We call 
$$\Degreey(F) =\max_{(\alpha,\beta) \in A} \ \alpha $$
the {\it  \degreey} of $F$ with respect to $y$ (symmetric definition for $\Degreez(F)$). 

\begin{Bem} Let $(\alpha,\beta)$ be the vertex of $N$ whose first component has the largest value among all vertices of $A$. Then the series $\widetilde F(y,z):=z^{-\beta} \cdot F(y,z) \in K[[y,z]]$ is regular with respect to the variable $y$, with pure monomial $y^\alpha$. By the Weierstrass Preparation Theorem $\widetilde F(y,z)$ equals, up to multiplication by a unit $U(y,z) \in K[[y,z]]^*$, a distinguished polynomial $P \in K[[z]][y]$ of degree $\alpha$ with respect to the variable $y$, i.e., $P(y,z)=U(y,z) \cdot \widetilde F(y,z)$, where $P=y^\alpha+c_1(z)y^{\alpha-1}+\ldots+c_{\alpha}(z)$, $c_i \in K[[z]]$, denotes a polynomial of order $\alpha$ at $0$ with respect to $y$. As $\alpha$ may be larger than the order of $\widetilde F$, the coefficients $c_i(z)$ may have order $<i$ at $0$. Up to multiplication by a unit in $K[[y,z]]^*$, also $F(y,z)=z^{\beta} \cdot \widetilde F(y,z)= U(y,z)^{-1} \cdot z^{\beta}  \cdot P(y,z)$ is a polynomial of degree $\alpha$ with respect to the variable $y$. Therefore it is 
justified to call $\alpha$ the degree of $F$ with respect to $y$. Note that we also have $\Degreey(F)=\Ordery(\widetilde F(y,0))$.

\end{Bem}


 \begin{figure}[h]
 \begin{center}
 \begin{minipage}{6cm}
 \beginpicture

 \setcoordinatesystem units <8pt,8pt> point at 0 0

 \setplotarea x from -1 to 32, y from -1 to 14
 \axis left shiftedto x=0 ticks numbered from 25 to 14 by 33
                                short unlabeled from 1 to 14 by 1 /
 \axis bottom shiftedto y=0 ticks numbered from 33 to 32 by 33
                                short unlabeled from 1 to 32 by 1 /

 \put {\circle*{5}} [Bl] at  14  12
 \put {\circle*{5}} [Bl] at  20  6
 \put {\circle*{5}} [Bl] at  30  3

 \plot 14 14  14 12 20 6  30 3  32 3  /   

 \vshade   14 12 14  <,z,,>  20 6 14 <,z,,> 30 3 14 <,z,,> 32 3 14 /

\put {\vector(0,1){96}} at 14  6

\put {\vector(0,-1){96}} at 14  6

\put {\small{\Degreey(F)}} at 16.3 1.5




\put {\vector(0,1){72}} at 8.5 7.5

\put {\vector(0,-1){72}} at 8.5 7.5

\put {\small{\Height(F)}} at 10.9 6.5

\put {\vector(0,1){24}} at 30  1.5

\put {\vector(0,-1){24}} at 30  1.5

\put {\small{\Ordery(F)}} at 28.2 1.5




\setdots <3 pt>

\plot 8.5 12   14 12  /

\plot 8.5 3   30 3  /



 \put {$z$} at 32 -0.8
 \put {$y$} at -0.8 14
 \endpicture

\end{minipage}
  \\[2ex]
 \begin{minipage}{11cm}
 \begin{abbildung}\label{Definition Hoehen} \begin{center} The values \Degreey(F), \Ordery(F) and \Height(F) of a polynomial
 $F$. \end{center} \end{abbildung}
 \end{minipage}
 \end{center}
 \end{figure}

\noindent We define the {\it height} of $F$ as
$$\Height(F) = \Degreey(F)-\Ordery(F).$$
This value clearly depends on the coordinates. It describes the vertical extension of the bounded edges of the Newton
polygon (see figure \ref{Definition Hoehen}) and will constitute (up to a correction term) the first component of our resolution invariant. 

\begin{Bem} Set $\alpha=\Ordery(F)$ and $\beta=\Orderz(F)$, so that $F$ factors into $F(y,z)=y^\alpha z^\beta\cdot H(y,z)$ with some polynomial $H$ which does not vanish identically along the two curves $y=0$ and $z=0$.  Then, clearly, $\Height(F) =\Degreey(F) - \Ordery(F)$ is just the order of $H(y,0)$ at $0$, i.e., the order at $0$ of the restriction of $H$ to the flag $\Fe$ defined by $z=0$.

\end{Bem} 

Analogously, we define the {\it width} of $F$ as $$\Width (F) = \textnormal{deg}_z (F) - \textnormal{ord}_z(F),$$
where $\textnormal{deg}_z(F)=\max_{(\alpha,\beta)\in A} \beta$ and $\textnormal{ord}_z(F)=\min_{(\alpha,\beta)\in A} \beta$. We call $F$ a {\it quasi-monomial} if $\Width (F)=1$ and $\Ordery(F)=0$. The respective singularities admit a simple resolution, see remark 12 in section \ref{nonincrease}. A similar notion appears in Hironaka's recent program for the resolution of singularities in characteristic $p>0$  \cite{HironakaTrieste}.

\ind If $N$ is a quadrant, we set the {\it \slope} of $F$ equal to $\Slope(F)=\infty$.
Otherwise, we define it as 
$$\Slope(F)=\frac{\alpha_1}{\alpha_1-\alpha_2} \cdot (\beta_2-\beta_1),$$ where $(\alpha_1,\beta_1)$ and
$(\alpha_2,\beta_2)$ denote those elements of $A$ whose first
components have the highest respectively second highest value among
all vertices of $A$
(see figure \ref{Projektion}). It is thus $-\alpha_1$ times the usual slope of the segment connecting the two points $(\alpha_1,\beta_1)$ and
$(\alpha_2,\beta_2)$. 
It will be the second component of our resolution invariant. \\
   

 \begin{figure}[h]
 \begin{center}
 \begin{minipage}{6cm}
 \beginpicture

 \setcoordinatesystem units <8pt,8pt> point at 0 0

 \setplotarea x from -1 to 22, y from -1 to 14
 \axis left shiftedto x=0 ticks numbered from 25 to 14 by 33
                                short unlabeled from 1 to 14 by 1 /
 \axis bottom shiftedto y=0 ticks numbered from 33 to 22 by 33
                                short unlabeled from 1 to 22 by 1 /

 \put {\circle*{5}} [Bl] at  4  12
 \put {\circle*{5}} [Bl] at  10  6
 \put {\circle*{5}} [Bl] at  20  3

 \plot 4 14  4 12 10 6  20 3  22 3  /   

 \vshade   4 12 14  <,z,,>  10 6 14 <,z,,> 20 3 14 <,z,,> 22 3 14 /

\setdots <3 pt>

\plot 4  12   10 6    16 0  /
\plot 4 12   4 0 /

\put {\vector(1,0){96}} at 10  0,7
\put {\vector(-1,0){96}} at 10  0,1
\put{$\Slope(F)$} at 9 0.8

\put{\small{$(\alpha_1,\beta_1)$}} at 6.3 12.5
\put{\small{$(\alpha_2,\beta_2)$}} at 12.3 6.5


 \put {$z$} at 22 -0.8
 \put {$y$} at -0.8 14
 \endpicture

\end{minipage}
  \\[2ex]
 \begin{minipage}{11cm}
 \begin{abbildung}\label{Projektion} \begin{center} \Slope(F) of
 $F$. \end{center} \end{abbildung}
 \end{minipage}
 \end{center}
 \end{figure}

\noindent As we will see in section \ref{Kap3.1}, especially in Lemma \ref{lem6} and in the example given in remark \ref{Bsp},  the \height \ can increase under blowup in some special situations. To correct this drawback, we will have to consider the position of the Newton polygon: Call $F$ {\it adjacent} if $\Ordery(F)=0$,
{\it close} if $\Ordery(F)=1$, and {\it distant} if $\Ordery(F)\geq 2$. The {\it bonus} of $F$ is set equal to
$$
\Bonus(F) = \left\{
\begin{array}{ll}
1+\delta & \textnormal{ if $F$ adjacent,}  \\
\varepsilon & \textnormal{ if $F$  close,} \\
0 & \textnormal{ if $F$ distant,} 
\end{array} \right.
$$ where 
$\delta$ and $\varepsilon$ denote arbitrary
constants $0 < \varepsilon< \delta < 1$.
Note that all these definitions break the symmetry
between $y$ and $z$.   Then we define the {\it \revheight\ }  of $F$ as $$\Revheight(F)=\Height(F)-\Bonus(F).$$


\ind We now associate these items in a coordinate independent
way to residue classes $f$ in $R/R^p$. For any choice of local
coordinates $(y,z)$ at $a\in \A^2$, take the unique expansion
$F=\sum_{\alpha\beta} c_{\alpha\beta}y^{\alpha} z^{\beta}$ of $f$ in
$K[[y,z]]$ with $(\alpha,\beta)\in \N^2\setminus p\cdot \N^2$. Let $\Fe$ be a local flag at $a$ fixed throughout,
and $\Ce=\Ce_\Fe$ the set of subordinate local coordinates $(y,z)$
in $R$ at $a$.  Note that the highest vertex $c=(\alpha,\beta)$
of $N=N(F)$ does not depend on the choice of the subordinate
coordinates, i.e., that any coordinate change subordinate to the
flag $\Fe$ leaves this vertex invariant. Hence $\Degreey(F)$ takes the same value for all subordinate coordinates. 
For $f\in R/R^p$ with expansion $F=F(y,z)$ at $a$ with respect to $(y,z)\in
\Ce$ we set \begin{eqnarray*} \height_a(f) &=& \min  \{\Height(F); (y,z)\in
\Ce\} \\ &=& \Degreey(F)-\max \{\Ordery(F); (y,z)\in \Ce\} \end{eqnarray*}
and call
it the {\it height} of $f$ at $a$. This number only depends on $f$, the
point $a$ and the chosen flag $\Fe$. 

\ind We say that $f$ is {\it monomial} at $a$ if there exists a (not necessarily subordinate) coordinate change transforming $F$ into a monomial
$y^\alpha z^\beta$ times a unit in $K[[y,z]]$. 
Note that this is in particular the case if
$\Height_a(f)=0$ (whereas the converse is not true). 

\begin{Bem} \label{bemerkung1}
A simple computation shows the following statement: If $f$ is adjacent and not monomial at $a$, then $\Height_a(f)$ is at least equal to $2$. 
\end{Bem}

\noindent By definition, $\Bonus(F)$ 
takes the same value, $\Bonus_a(f)$, for all coordinates realizing $\height_a(f)$, because
$\Ordery(F)$ does.  We conclude that \begin{eqnarray*} \Revheight_a(f)&=&
\height_a(f)-\Bonus_a(f) \\ &=&\min\{\Height(F)-\Bonus(F); (y,z)\in
\Ce\}\end{eqnarray*} only depends on $f\in R/R^p$, the point $a$ and the chosen
flag $\Fe$. This will be the first component of our local resolution invariant.
It belongs to the well ordered set $\N_{\delta,\varepsilon}=\N - \delta \cdot \{0,1\} - \varepsilon \cdot \{0,1\}$. 
As we mostly consider fixed points we omit the reference to $a$ 
and simply write 
$$\Revheight(f)=\height(f)-\Bonus(f).$$ 
%
%
\ind The second component of our local resolution invariant is given by $$\Slope_a(f)=\max\{\Slope(F); (y,z) \in \Ce \textnormal{ with } \Height(F)=\height_a(f)\}.$$ It is called the {\it \slope} of $f$ and
also only depends on $f\in R/R^p$, the point $a$ and the chosen flag
$\Fe$. Again we omit the reference to $a$ and simply write
$\Slope(f)$. \\

\ind The local resolution invariant of $f\in R/R^p$ at $a$ with respect to the chosen flag $\Fe$ is defined as
$$i_a(f) = (\Revheight_a(f), \Slope_a(f)).$$ We consider this pair
with respect to the lexicographic order with $(0,1) < (1,0)$, and call it the {\it adjusted height vector invariant} of $f$ at $a$. Sometimes we shall write $i_a(f,\Fe)$ in order to emphasize the dependence on the flag.\\

\begin{Bem} \label{comp}
Note that $\Height_a(f)$ and $\Slope_a(f)$, which are the main ingredients of our local resolution invariant, and the primary measure $\textnormal{ord}_b(G)$ are all of the same type (see figure \ref{diagram_components}):  The order of $G(x,y,z)$ at a point $b$ equals in the purely inseparable case $G=x^p+F(y,z)$ with $\textnormal{ord}(F) \geq p$ the height of the Newton polyhedron $N(G) \subset \N^3$ with respect to the variable $x$, cf. figure 5. Furthermore the coefficient ideal of $G(x,y,z)$ with respect to $x=0$ is generated by the
polynomial $F(y,z)$. And $\Height(f)$ measures the (minimal) height of the Newton polygon $N(F) \subset \N^2$ of the polynomial $F$ with respect to the variable $y$. 
Finally the \slope \ of $F$ can be thought of as a certain height 
of the Newton polygon in $\N$ of the coefficient ideal of $F$ in $y=0$.   \\

	\begin{figure}[h]
	\begin{center}
	\includegraphics[scale=0.45]{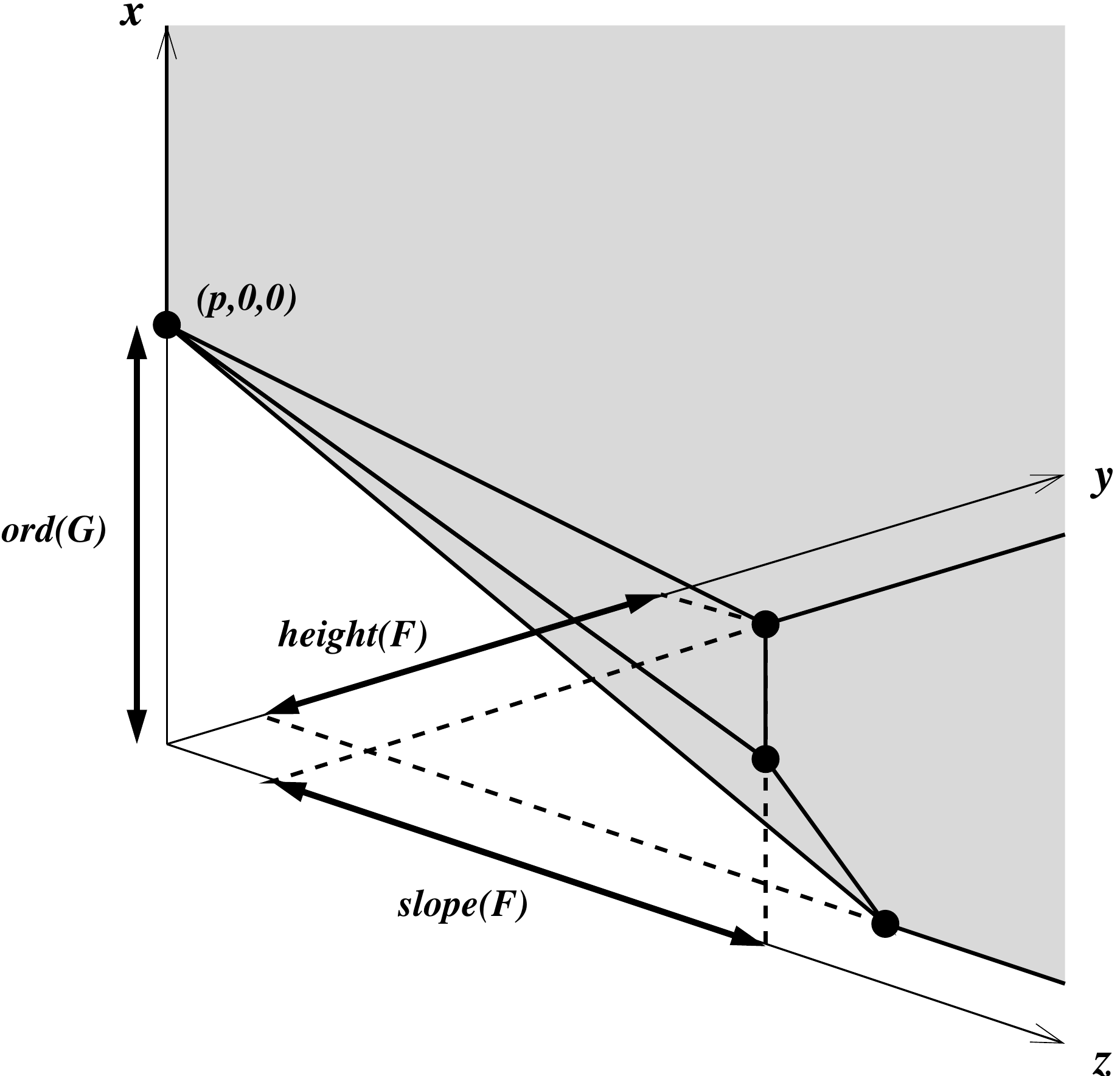}   
	\\ 
	[1ex]  \begin{minipage}{11cm}
       \begin{abbildung}\label{diagram_components} \begin{center}
       The measures $\textnormal{ord}(G)$, $\Height(F)$ and $\Slope(F)$. \\ \end{center} \end{abbildung}
      \end{minipage}
	\end{center}
	\end{figure}

\end{Bem}

\begin{Bem} \label{abhyankar}
Apparently, Abhyankar has considered an invariant which is similar to our $i_a(f)$, see section 7 in \cite{Abhyankar-nonsplitting} and definition 7.3 in \cite{Cutkosky-on-Abhyankar}. He treats the more general case of arbitrary monic polynomials of order a power of the characteristic. In the case of purely inseparable polynomials of order $p$, his case distinction is very close in spirit to ours (though there are more cases in \cite{Abhyankar-nonsplitting}); his bonus, however, is slightly smaller (equal to $1$ if $F$ is adjacent, and $0$ otherwise), which implies that the invariant drops less often (but the definition still ensures that it does not increase). The case of constancy of the resolution invariant may therefore occur more often; this requires to take into account additional invariants to show that the situation improves under blowup.
\end{Bem}

\begin{Bem} \label{hironaka}
There also appears a similarity to the invariant of Hironaka  \cite{Hironaka, MR1748627}: Let $g$ be an element of the coordinate ring $S$ of $\A^3$ and let $G=\sum_i c_i(y,z)x^i$ be the expansion of $g$ with respect to a local coordinate system $(x,y,z)$ at the point $b \in \A^3$. Let $d$ be the order of $g$ at $b$. After a generic linear coordinate change we may assume that $c_d(0,0) \neq 0$. Due to the Weierstrass Preparation Theorem there exists an invertible power series $u(x,y,z)$ such that $u \cdot g =x^d+\sum_{i<d}c_i'(y,z)x^i$. Now let $N_{y z}(G) \subseteq \Q_+^2$ be the projection with center $(d,0,0)$ of the Newton polyhedron $N(G)$ of $G$ onto the $yz$-plane. Note that this projection extends the induction argument to arbitrary surface equations and their associated first coefficient ideal instead of dealing only with purely inseparable surface equations. Let $\alpha=(\alpha_y,\alpha_z)$ and $\beta=(\beta_y,\beta_z)$ denote those vertices of $N_{yz}(G)$ whose $y$-components have the largest respectively second largest value among all vertices of $N_{yz}(G)$ (in case that $G$ is not monomial). Furthermore let $s_{\alpha \beta}=\frac{\alpha_y-\beta_y}{\alpha_z-\beta_z} \in \Q_-$ be the usual slope of the segment from $\alpha$ to $\beta$. Then Hironaka defines the following rational vector 
$$J_{b,(x,y,z)}(G)=(\textnormal{ord}_a(g),\alpha_y,s_{\alpha \beta},\alpha_y+\alpha_z)\in \Q^4.$$ 
To obtain a coordinate free definition,  choose subordinate coordinates for the chosen flag which maximize $(\alpha_z,\alpha_y,s_{\alpha \beta},\alpha_y+\alpha_z)$ with respect to the lexicographic order. Now the resolution invariant is given as $i_b=i_{b,(x,y,z)}=(\textnormal{ord}_a(g),\alpha_y,s_{\alpha \beta},\alpha_y+\alpha_z)$, where $(x,y,z)$ denote such maximizing subordinate coordinates. 
Additionally to the difference in the definition of the invariants in the approach of Hironaka and ours, another crucial distinction lies in the choice of admissible centers. More precisely, whereas we always blow up in a point until $G$ is of the form $G=x^p+y^mz^nA(y,z)$ with $A(0,0) \neq 0$ (and then also allow smooth curves as centers) or until $\textnormal{ord}(G)$ has dropped, Hironaka has to distinguish in each step whether the top locus of $g$ contains a smooth curve or just consists of isolated points. In order to show the decrease of the invariant, he then has to choose the largest possible smooth center. This restriction makes it difficult  to generalize the method and the invariant to higher dimensions. \\
\end{Bem}

\begin{Bem}

The resolution invariant could also be defined differently: First, instead of adding $\varepsilon$ and $\delta$ to the bonus, one could take the intricacy minus $0$ or $1$, but add to it a new component, which captures the adjacency. Its value would be $-2$ for being adjacent, $-1$ for being close, and $0$ otherwise. The treatment of quasi-monomials would have to be revised.

Secondly, one could take the minimum of the height over all coordinates, and not just those subordinate to a flag. In fact, the minimum has always to be realized before the blowup, so the restriction to triangular coordinate changes is of no help. However, for the slope (which is a maximum), subordinate coordinates may be necessary. Observe here that for arbitrary polynomials of order $p$, one has to take first a maximum to define a suitable coefficient ideal and then its height. For instance, one could first maximize the exceptional factors and then take the resulting height as a minimum.

Thirdly, back in the purely inseparable case, if we assume having factored an exceptional monomial from $F$, one could take the height as the minimum over all coordinate changes. But then, applying a generic linear change $(y,z)\rightarrow (y,z+ty)$ we would always fall back on the order (say, in the terminology of Hironaka, the residual order). So it seems that the height is not so far away from the order of the coefficient ideal.

Finally, let us comment on the choice of the bonus: The crucial datum which bounds the height or the order after blowup under a translational move is the ``volume'' of the initial form $F_d$ of $F$, say the number $d-r-s$ where $y^rz^s$ denotes the maximal monomial which can be factored from $F_d$ in suitable coordinates (not from $F$). By the characterization of wild singularities we know that an increase of the height or order can only occur when $d$ is a multiple of the characteristic. In this case, $r=0$ does not occur, since $z^d$ is a $p$-th power and does not count. But if $F_d$ is not a monomial (and still $d$ is a multiple of $p$), also the monomial $yz^{d-1}$ can be eliminated by a linear coordinate change, so that the volume of $F_d$ is always $\leq d-2-s$ if a preliminary $r$ would be $0$ or $1$ (say, if $F_d$ in given coordinates is adjacent or close). This argument is not valid if $d$ is not a multiple of $p$, but in this case no increase of the height or order occurs (nevertheless, if the height or the order remain constant, a secondary measure has to be considered and shown to decrease. This is still an open question.)

All this suggests to measure in the bonus the difference of the height/order of the initial form $F_d$ and the height/order of $F$. Here, the congruence of $d$ modulo $p$ comes into play, as well as the location of $F_d$ with respect to the $z$-axis.

\end{Bem}\medskip


\section{Logical structure of the proof of Theorem \ref{Theo1}} \label{Kap2} 
$ $ \\
\noindent We sketch in this section the reasons for the decrease of the adjusted height vector under point blowup, i.e., the proof of Theorem \ref{Theo1} (the details come in the next section). Due to the definition of the invariant, this will immediately imply the local monomialization of $F(y,z)$ modulo $p$-th powers, from which there is an easy combinatorial way to decrease the order of the purely inseparable surface equation $G=x^p+F(y,z)$ by finitely many further point and curve blowups (section \ref{Kap5}). Together with the study of the non-monomial locus in section \ref{termination}, this will also establish Corollary \ref{Theo2}.

Before explaining the overall strategy we specify the statement of Theorem \ref{Theo1}. Let $a$ be a closed point of $\A^2$ and let $\Fe$ be a fixed local flag in $\A^2$ at $a$. Let $\pi:\widetilde \A^2\rightarrow \A^2$ be the blowup with center $a$ and exceptional divisor $E =\pi^{-1}(a)$. 
The flag $\Fe$ at $a$ induces in a natural way a flag $\Fe'$ at any closed point $a'$ of $E$ by setting
$$\Fe'=\left\{ \begin{array}{lll} \Fe^s & \textnormal{if} & a' \in E \cap \Fe^s, \\ E & \textnormal{if} & a' \notin E \cap \Fe^s,\end{array}\right.$$
where $\Fe^s$ denotes the strict transform of $\Fe$ under $\pi$ (for more details see \cite{MR2063657}).

\ind Denote by $R'$ the respective Rees algebra of the coordinate ring $R$ of $\A^2$, say $R'=\oplus_{k\geq0}\, \mathfrak{m}^k$, where $\mathfrak{m}$ denotes the maximal ideal of $R$ defining $a$. Denote by $f' \in R'/ R'^p$
the strict transform of $f$ under $\pi$, defined as the equivalence class of the  strict transform $F'$ of a representative $F$ of $f$. It is a simple task to check that $f'$ is well defined, i.e., does not depend on the various choices. Thus we dispose of the adjusted height vector $i_{a'}(f',\Fe')$ of $f'$ at all points $a'$ of $E$. Theorem \ref{Theo1} then reads as follows. 


\noindent 
\begin{Theo} \label{Theo11}  (i) Let $F$ be a polynomial in two variables $y, z$ over an algebraically closed field of characteristic $p>0$. Denote by $f$ the residue class of $F$ modulo $p$-th powers and assume that $f$ is not a quasi-monomial at a given closed point $a$ of $\A^2$. Fix a local flag $\Fe$  in $\A^2$ at $a$. Let $\tau: \widetilde{\A^2}\rightarrow \A^2$ be the point blowup with center $a$ and exceptional divisor $E=\pi^{-1}(a)$.  For any closed point $a'$ in $E$, denoting by $f'$ and $\Fe'$ the transforms of $f$ and $\Fe$  in $\widetilde{\A^2}$, the adjusted height vector $i_a(f,\Fe)$ of $f$ at $a$ with respect to $\Fe$ satisfies $$i_{a'}(f',\Fe') < i_a(f,\Fe).$$ 

\noindent (ii) Let $X$ be a reduced two-dimensional closed subscheme of a smooth three-dimensional ambient scheme $W$ of finite type over an algebraically closed field of characteristic $p>0$. Let $b$ be a singular closed point of $X$ of order $p$. Assume that $X$ is defined in local coordinates of $W$ at $b$ by a purely inseparable equation of the form $$G(x,y,z)=x^p+F(y,z). $$ 

\ind Finitely many point blowups transform $X$ into a scheme which, locally at any point  of order $p$ above $b$, can be defined by an equation $G(x,y,z)=x^p+F(y,z)  $ with $F$ a monomial. 
\end{Theo}


 \begin{Bem} \label{strict} It is easy to see that it doesn't make any
difference in proving Theorem \ref{Theo1}  if we work with the strict transform $f'$ or the total transform $f^*$ of $f$ under the point blowup $\pi$, because their Newton polygons differ just in a displacement by $p$ units in either the $y$- or the $z$-direction
(depending on the point $a'$ of $E$). The measure \height \ is hence the same for both transforms. 
Moreover such a displacement may only increase the adjacency and consequently  decreases the \revheight, i.e., $\Revheight(f') \leq \Revheight(f^*)$. Furthermore the measure \slope \ is, as we will see, only needed in the horizontal move (see  below) and in this situation $\Slope(f')=\Slope(f^*)$ holds.
And since computations are
simpler when using the total transform $f^*$, we will show that $i_{a'}(f^*) < i_a(f)$,
which then immediately implies $i_{a'}(f')<i_a(f)$. \\
\end{Bem}

\begin{Bem} \label{Bem4} 
The transformation of the equation of our original surface $G(x,y,z)=x^p+F(y,z)$ under blowup of $\A_K^3$ in a point $b=(b_1,a)$, fulfilling $\textnormal{ord}_b(G)=p$, can be read off from the transformation rule for $F$
under the point blowup $\pi$ of $\A_K^2$ in $a$ as follows: Let
$(y,z)$ and $(x,y,z)$ be regular parameter systems of the local
rings $\widehat R_a$ and $\widehat S_{b}$ of $\A_K^2$ at $a$ and $\A_K^3$ at $b$.
Furthermore, denote by $F(y,z)$
respectively $G(x,y,z)$ the expansions at $a$ and $b$ of elements $f \in R/R^p$ and $g
\in S$ with respect to the chosen local coordinates. With $g^*$ and $g' \in S'$ we denote the total respectively strict
transform of $g \in S$, where $S'$ denotes the
Rees-algebra of $S$ corresponding to the blowup of $\A_K^3$ in $b$.
The chart-expressions for the total transform of $G$ under the blowup $\tau: \widetilde \A_K^3 \rightarrow \A_K^3$
with center $b=0$ look as follows: 
\begin{itemize}
\item[] $x$-chart: $G^*(x,y,z)= x^p \cdot (1+x^{-p}F(xy,xz)),$
\item[] $y$-chart: $G^*(x,y,z)=y^p \cdot (x^p+y^{-p}F(y,yz)),$
\item[] $z$-chart: $G^*(x,y,z)=z^p \cdot (x^p+z^{-p}F(yz,z)).$
\end{itemize}  
The $x$-chart can be discarded since the strict transform of the surface does not pass through its origin which is the only point of the $x$-chart not contained in the $y$- or $z$-chart. In the $y$-chart
(and symmetrically in the $z$-chart) either $\textnormal{ord}(g') <
\textnormal{ord}(g)=p$ and we are done, or $\textnormal{ord}(g')=\textnormal{ord}(g)=p$, say
$\textnormal{ord}(y^{-p}F(y,yz))
\geq p$, hence $G'(x,y,z)=x^p+y^{-p}F(y,yz)$ is of the same
type as $G$. Since multiplying $F(y,yz)$ by $y^{-p}$ again has
only the effect of a displacement when regarding the
corresponding Newton polygons, it is sufficient to study the total transform of $F$ under the
blowup $\pi$ of $\A_K^2$ with the two chart
expressions $F^*(y,z)=F(y,yz)$ and $F^*(y,z)=F(yz,z)$.  \\
\end{Bem} 


\ind Fix subordinate coordinates $(y,z) \in \Ce_{\Fe}$ at the closed point $a \in \A^2_K$ realizing the \height \ of $f$, i.e., satisfying $$\Height(F)=\Height(f),$$ where $F(y,z)$ denotes the expansion of $f \in R/R^p$ with respect to $y$ and $z$. Let $a' \in E=\pi^{-1}(a)$ be a point above $a$. There then exists a unique constant $t\in K$
such that the blowup $\widehat R_a\rightarrow \widehat R'_{a'}$ is given either 
by $(y,z)\rightarrow (yz+tz,z)$ or $(y,z)\rightarrow (y,yz)$.
Accordingly, and distinguishing
between $t=0$ or not, $f^*$ has expansion $F^*$ in $\widehat
R'_{a'}\cong K[[y,z]]$, where $(y,z)$ now denote local coordinates subordinate to the induced flag $\Fe'$ at $a'$,  given by the following formulas: \medskip
\begin{itemize} 

\item[] (T) Translational move: $F^*(y,z)=
F(yz+tz, z)$, $t\in K^*$,\medskip

\item[] (H) Horizontal move: $F^*(y,z)=F(yz,z)$,\medskip

\item[] (V) Vertical move: $F^*(y,z)=F(y,yz)$. \medskip
\end{itemize} 
\rem{These formulas are compatible with the flags at
$a$ and $a'$ in the sense that subordinate coordinate changes are still flexible enough to render local blowups monomial.  This is shown in more generality in \cite{MR2063657}. \\ \\}
The naming of the moves (H) and (V) stems from the corresponding transformations of the
Newton polygons.  Note that there could be several different subordinate coordinates in $\Ce_{\Fe}$ realizing the \height \ of $f$.
If possible, we will choose among all these minimizing subordinate coordinates  a pair $(y,z) \in \Ce_{\Fe}$ in which the blowup $\widehat{R}_a \rightarrow \widehat R'_{a'}$ is {\it monomial}
(moves (H) and (V)).  

The subtlety of the proof that the adjusted height vector drops under blowup for all points $a' \in E$ is due to the fact that the three moves change the Newton polygon in pretty different ways. The invariant has to drop lexicographically under {\it all} these moves. The key ingredients for this are the following:

Under translational moves, the height can at most increase by $1$ (by Moh's bound), and if it does, the Newton polygon was not adjacent before the blowup (by Hauser's description of the kangaroo phenomenon), but must be adjacent afterwards (by the definition of the height). 

Under horizontal moves, the height cannot increase (because the vertices of the Newton polygon move horizontally), the adjacency remains the same (for the same reason). Moreover, in a sequence of horizontal moves, the height must eventually drop (because the slope decreases in each move for which the height remains the same). 

Under vertical moves, the height decreases at least by $2$ (by a simple computation, with the exception of quasi-monomials), and the polygon may quit being adjacent or close.

From  these observations it is straightforward to see how the bonus has to be defined (in dependence of the adjacency) so that one obtains a decrease of the adjusted invariant under each blowup: Take value $0$ for $f$ distant, $\varepsilon$ for $f$ close, $1+\delta$ for $f$ adjacent, with $\varepsilon < \delta$. This choice yields an adjusted height vector that interpolates the ``graph'' of the original height vector over a sequence of blowups by a strictly decreasing function.  Induction applies! 
  
Let us see this argument in more detail. Let $a$ and $a'$ be fixed, and recall  that $\Revheight(f)=\Height(f)-\Bonus(f)$ (see section \ref{Kap1}). If there don't exist  subordinate coordinates at $a$ realizing the \height \ of $f$ and so that the blowup is monomial (i.e., if the translational move (T) is {\it forced}), then one always has $$\Revheight(f^*)<\Revheight(f),$$ where $f^*$ denotes the equivalence class modulo $p$-th powers of the transform $F^*(y,z)=F(yz+tz,z)$ with $t \neq 0$.  

\ind Next assume that one can choose subordinate coordinates $(y,z)$ at $a$ realizing the \height \ of $f$ such that $a'$ is one of the two origins of $\widetilde{\A }^2$, say cases (H) or (V) given by monomial substitutions occur. We then have

\begin{eqnarray*}
\Revheight(f^*)=\Height(f^*)-\Bonus(f^*)  &\leq& \Height(F^*)-\Bonus(F^*) \\ &\leq& \Height(F)-\Bonus(F) \\ &=& \Revheight(f),
\end{eqnarray*}
except for very special situations where $f$ is a quasi-monomial (these can be resolved directly, see section \ref{vertical-move}). Moreover, excluding these exceptions, the inequality is strict for move (V).
In case of equality
$$\Revheight(f^*)=\Revheight(f)$$ when applying move (H), we use the second component of the
invariant and show first that $$\Slope(F^*)<\Slope(F) \leq \Slope(f).$$
Realizing
$\Slope(f^*)$ is by definition done by maximizing $\Slope(F^*)$ over all coordinate choices subordinate to the flag
$\mathcal{G}$ at $a'$. It has to be shown that the necessary coordinate change $\varphi'$ at $a'$
stems from a coordinate change $\varphi$ at $a$ subordinate to $\Fe$ (see section \ref{maximieren}). Or said differently, one has to prove that the following diagram commutes, where $(y,z)$ denote local subordinate coordinates at $a$ and where the blowup $\pi: \widehat R_a \rightarrow \widehat R'_{a'}$ is given by $(y,z) \rightarrow (yz,z)$ (inducing subordinate local coordinates to the flag $\Fe'$ at $a'$ on $R'_{a'}$)

\begin{equation*}\begin{CD} 
\widehat R'_{a'} @>{\varphi'}>> \widehat R'_{a'}  \\ 
@A{\pi}AA @A{\pi}AA \\ 
 \widehat R_a @>{\varphi}>> \widehat R_a\\  \\
\end{CD}\end{equation*} 

\noindent and $\varphi'(y,z)=(y+A(z),z)$, $\varphi(y,z)=(y+A(z) \cdot z,z)$ with $A \in K[[z]]$.

\ind The general behavior of the \height \ is illustrated in figure \ref{diagram-height}: it may increase in one step (but only under translational moves), but decreases in the long run of the resolution process. \\

	\begin{figure}[h]
	\begin{center}
	\includegraphics[scale=0.35]{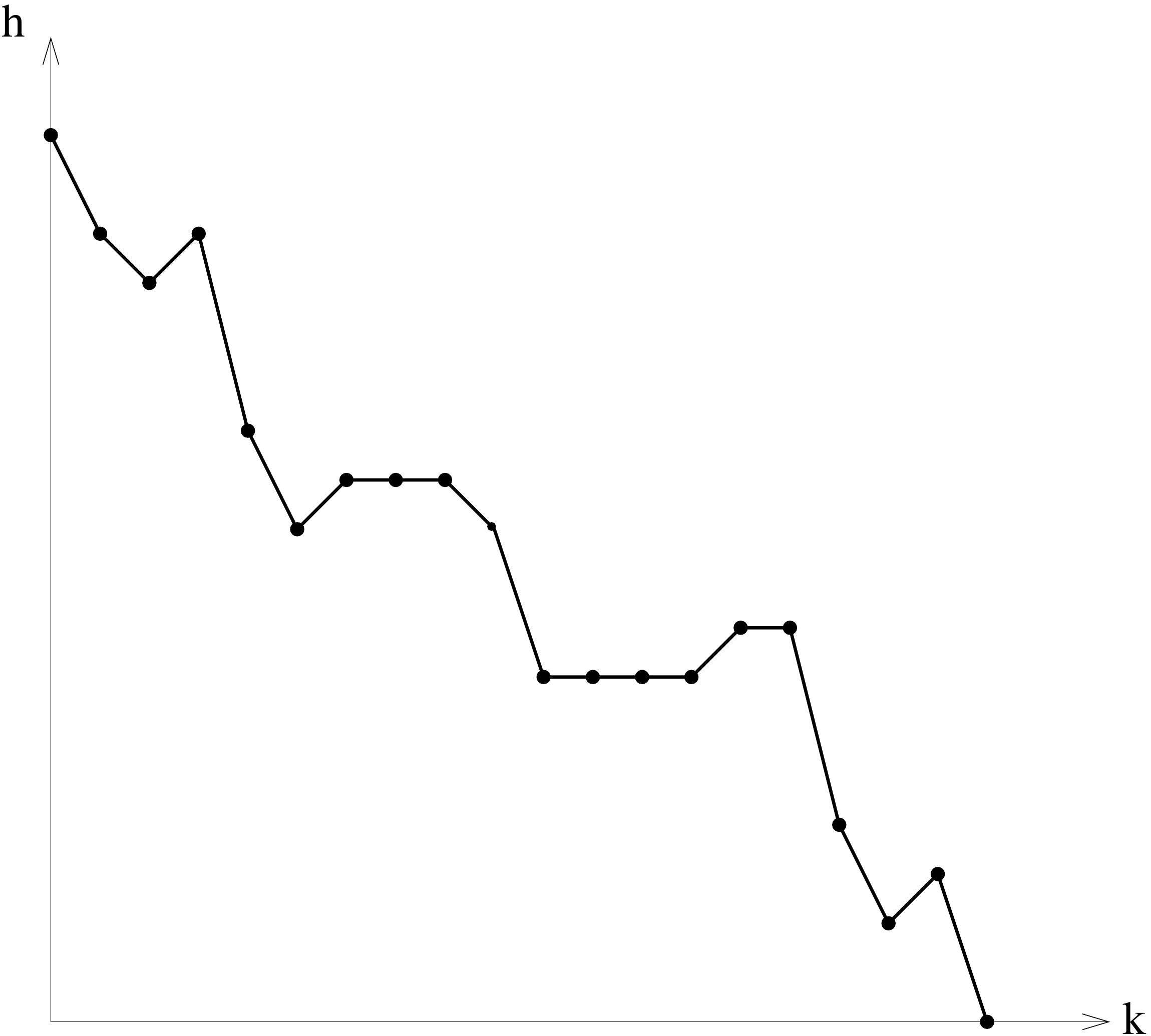}   
	\\ 
	[1ex]  \begin{minipage}{12cm}
       \begin{abbildung}\label{diagram-height} \begin{center}
       Possible behavior of the \height \ under blowups ($h$ denotes $\height(f)$, $k$ the number of blowups); an increase can only occur under translational moves. \\ \end{center} \end{abbildung}
      \end{minipage}
	\end{center}
	\end{figure}


\section{Proof of Theorem \ref{Theo1}} \label{Kap3}
$ $ \\
\noindent We show in section \ref{nonincrease} that the first component  of the adjusted height vector $i_a(f)=(\Revheight(f), \Slope(f))$ does not increase under point blowup (except for quasi-monomials). In section \ref{maximieren} it is shown that if the intricacy remains the same, the second component of $i_a(f)$  decreases.\\ 

\subsection{Non-increase of the \revheight} \label{nonincrease}
$ $ \\

\noindent The key argument in proving Theorem \ref{Theo1} is the following:

\begin{Prop} \label{intricacy}
Let $f$ be an element of $Q=R/R^p$, which is not a quasi-monomial at a given closed point $a$ of $\A^2$. Fix a local flag $\Fe$  in $\A^2$ at $a$ and denote by $\Fe'$ the induced flag at $a'\in E$. Let $F \in K[[y,z]]$ be the expansion of $f$ with respect to subordinate coordinates $(y,z) \in \Ce_{\Fe}$ realizing the height  of $f$. Furthermore let $F^*(y,z)$ be one of the transformations $F^*(y,z)=F(yz+tz,z)$, with $t\in K$, or $F^*(y,z)=F(y,yz)$ and let $f^*$ be the corresponding element in $R'/R'^p$. Then $$\Revheight(f^*) \leq \Revheight(f).$$ Moreover, if the translational move (T) is forced, or if there exist subordinate coordinates realizing the height of $f$ such that the blowup $\widehat {R}_a \rightarrow \widehat R'_{a'}$ is given by move (V), then $$\Revheight(f^*) < \Revheight(f).$$   
\end{Prop} 

\noindent Recall that adjacent series $F$ with $\Width(F)=1$ are called quasi-monomials. Quasi-monomials are not resolved directly, but if $F$ is such, the order of $G$ is decreased by line blowups.
Note that by the minimality of the height, there is no need to realize the \height \ of $f^*$ in $R'/R'^p$. 
The proof of Proposition \ref{intricacy} falls naturally into three parts corresponding to the
three different moves (T), (H), (V) defined in section  \ref{Kap2}.  \\


\subsection*{(T) Translational moves}  \label{Kap3.1}
$ $ \medskip

\noindent The goal of this paragraph is to show Proposition \ref{intricacy} for the translational move $F^*(y,z)=F(yz+tz,z)$ with $t \in K^*$.  In particular we prove:  the \revheight \ decreases if there don't exist minimizing subordinate coordinates such that $a' \in E$ is one of the origins of the two charts of the blowup. 
Since situations where a translational move is required are the most delicate ones, this section provides the main arguments for proving Theorem \ref{Theo1}.

\ind In the following $d=\textnormal{ord}(f)$ denotes the order of $f$ in $a$ and $f_d$ its initial form.
Furthermore  the {\it parity} $\Parity(d)$ of $d$ is set as $1$ if
$d \equiv 0 \textnormal{ mod } p$, and $0$ otherwise.

\ind In the sequel it will be assumed throughout that there don't exist subordinate coordinates at $a$ realizing the \height \ of $f$ such that the blowup $\widehat R_a \rightarrow \widehat R'_{a'}$ is monomial. Or said differently, there don't exist minimizing subordinate coordinates at $a$ such that $a' \in E$ is one of the origins of the two charts of the blowup. In this situation the total transform $f^*$ of $f$ under the blowup $\pi$ is given as the equivalence class of the transform $F^*(y,z)=F(yz+tz,z)$, where $t \in K^*$, of a representative $F(y,z)$ of $f$ with $\Height(F)=\Height(f)$. Fix such minimizing subordinate coordinates $(y,z) \in \Ce_{\Fe}$ and denote by $F(y,z)$ in the sequel always the expansion of $f$ with respect to these chosen coordinates.

\begin{Bem} \label{Ordery}
It can be easily verified that the situation $\textnormal{ord}_y(F^*) > \textnormal{ord}_y(F)$ cannot occur. This is due to the fact that the transformation $(y,z) \rightarrow (yz+tz,z)$ with $t \neq 0$ can be decomposed into a linear subordinate coordinate change $(y,z) \rightarrow (y+tz,z)$ followed by a horizontal move $(y,z) \rightarrow (yz,z)$. Due to the minimality of $\Height(F)$ (which corresponds to the maximality of $\Ordery(F)$), the first one does not increase the order with respect to the variable $y$. The second transformation clearly preserves it.  

Moreover, in the case that $\textnormal{ord}_y(F^*)=\textnormal{ord}_y(F)$ the same argumentation shows that there exist subordinate coordinates realizing the \height \ of $f$ such that the blowup can be rendered monomial. By the assumption at the beginning of this section, one would thus choose these new minimizing coordinates and would hence be left with the examination of a horizontal move (see subsection (H) below). 
 
Altogether this shows that for the study of translational moves it suffices to investigate the situations where $\textnormal{ord}_y(F^*)<\textnormal{ord}_y(F)$. \\
\end{Bem}

\ind The proof of Proposition \ref{intricacy} in the case of translational moves is divided into a series of lemmata. Recall that the {\it adjacency} $\Adj(F)$ of $F$ is $2$, $1$ or $0$ according to $F$ being {\it adjacent}, $\Ordery(F)=0$,
{\it close}, $\Ordery(F)=1$, or {\it distant}, $\Ordery(F)\geq 2$. By definition, $\Adj(F)$  takes the same value, $\Adj(f)$, for all coordinates realizing $\height(f)$, because $\Ordery(F)$ does. 

\begin{Lem} \label{degreey}
Every $F$ satisfies 
$$\Height(F) \leq \textnormal{deg}_y(F) -2 + \Adj(F).$$
\end{Lem}

\begin{proof} 
This is clear from the definitions.  
\end{proof}


\ind The next result is due to Moh (cf. Proposition 2, p. 989 in \cite{MR935710}, or Theorem 3 in \cite{Hausera}):


\begin{Lem} \label{lem5}
Let $F_d$ be homogenous of degree $d$. Set $F_d^+(y,z)=F_d(y+tz,z)$ with $t
\neq 0$. Then
$$\Ordery(F_d^+) \leq \Height(F_d) + \Parity(d).$$
\end{Lem}


\begin{proof} 
 (a) First we consider the case $\Parity(d)=1$. Let $F_d$ have
$\Height(F_d)=k$ and represent it as
$$F_d(y,z)=\sum_{i=0}^k c_i y^{m-i}z^{n+i}$$ with $c_i \in K$, $c_0, c_k \neq
0$, $k,m,n \in \N$, $k \leq m$ and $m+n=d$. Set $v=\Ordery(F_d^+)$.  \\
First observation: The term $y^{m-k}z^n$ divides $F_d$, hence  $F_d^+ \in z^n \langle y+tz \rangle
^{m-k}$. By assumption $m+n \in p \cdot \N$ and $(m-k,n+k) \notin p
\cdot \N^2$, which implies $m-k \notin p \cdot \N$. Therefore
$$\partial_y F_d^+ \in z^{n} \langle y+tz \rangle^{m-k-1}.$$ Second
observation: There exists a polynomial $D$ with $D(0,z) \neq 0$ and
$$F_d^+(y,z)=y^v \cdot D(y,z).$$ Since $v \notin p \cdot \N$ (otherwise the monomial
$y^vz^{d-v}$ occurring in the expansion of $F_d^+$ would be a $p$-th power and thus $\textnormal{ord}_y(F_d^+)>v$),
it
follows that $$\partial_y F_d^+ = vy^{v-1} D(y,z) + y^v
\partial_y D(y,z) \neq 0,$$ and therefore $$\partial_y F_d^+ \in
\langle y \rangle ^{v-1}.$$ Combining these two observations leads
to
$$\partial_y F_d^+ \in z^{n} \langle y+tz \rangle^{m-k-1} \cap
\langle y \rangle ^{v-1}.$$ But $t \neq 0$
and thus $$\partial_y F_d^+ \in z^{n} \langle y+tz \rangle^{m-k-1}
\cdot \langle y \rangle ^{v-1}.$$ Since $\textnormal{ord}(F_d^+)=m+n$
and $\partial_y F_d^+ \neq 0$ it follows that
\begin{eqnarray*}
n+m-k-1+v-1 & \leq & m+n-1.
\end{eqnarray*}
Hence $v \leq k+1$ as required. \\ \noindent (b) In the same manner as in (a)
one can see that in the case $\Parity(d)=0$ one gets $F_d^+ \in z^{n}
\langle y+tz \rangle ^{m-k}$ and $F_d^+ \in \langle y \rangle ^v$.
Combining this and using $t \neq 0$ results in $$F_d^+ \in z^{n}
\langle y+tz \rangle ^{m-k} \cdot \langle y \rangle ^v.$$ From this
it follows that $v \leq k$. \\
\end{proof}



\begin{Lem}\label{lem6} Let $F^*(y,z)=F(yz+tz,z)$ with $t
\neq
0$ and $d = \textnormal{ord}(F)$. Then \\
$$\Degreey(F^*) \leq \Height(F_d)+\Parity(F). $$
\end{Lem}


\begin{proof}
Write $F$ as $F(y,z)=F_d(y,z)+H(y,z)$ with $H \in K[[y,z]]$ and $\textnormal{ord}(H) >d$. Furthermore represent $F_d$ as in the proof of Lemma \ref{lem5}. Since $t \neq 0$ one gets for $F^*$
\begin{eqnarray*}
F^*(y,z) &=& \sum_{i=0}^k c_i (yz+tz)^{m-i}z^{n+i}
+ H(yz+tz,z) \\ &=& \underbrace{z^d \cdot \sum_{j=0}^m c_j' y^j}_{=:A(y,z)}  +
\underbrace{H(yz+tz,z).}_{=:B(y,z)} \\
\end{eqnarray*}
It is obvious that $\textnormal{ord}(A) \geq d=m+n$ and
$\textnormal{ord}_z(B)
> d$.
Moreover the last lemma implies $\Ordery(A) \leq
\Height(F_d)+\Parity(d)=k+\Parity(d)$. Therefore there exists an integer $j \in \{0,1,\ldots,k+\Parity(d)\}$ such that
$c_j' \neq 0$. Let $l$ be the smallest. Then $A$ can be written as
$A(y,z)=z^d \sum_{j=l}^m c_j'y^j.$ It follows that
$$\Degreey(F^*)=l \leq k+\Parity(F)=\Height(F_d)+\Parity(F).$$
\end{proof}

       
\begin{Bem} \label{Bsp}
The inequality of the previous lemma is sharp! Take for example $p=2$
and $F(y,z)=y^5z+y^3z^3+y^3z^8$. Then we have $\Height(F_d)=2$ and
$F^*(y,z)=F(yz+1 \cdot z,z)$ with $\Degreey(F^*)=3$ (see
figure \ref{Lemma5}). 
\end{Bem}

	 

\begin{figure}[h]
\begin{center}
\begin{minipage}{6cm}
\beginpicture

 \setcoordinatesystem units <9pt,9pt> point at 0 0

 \setplotarea x from -1 to 13, y from -1 to 11
 \axis left shiftedto x=0 ticks numbered from 5 to 11 by 5
                                short unlabeled from 1 to 11 by 1 /
 \axis bottom shiftedto y=0 ticks numbered from 5 to 13 by 5
                                short unlabeled from 1 to 13 by 1 /

 \put {\circle*{5}} [Bl] at  1  5
 \put {\circle*{5}} [Bl] at  3  3
 \put {\circle*{5}} [Bl] at  8  3

 \plot 1 11  1 5  3 3  13 3 /

 \vshade 1 5 11 <,z,,> 3 3 11 <,z,,>  13 3 11 /

 \multiput {\multiput {\circle{5}} [Bl] at 0  0 *6 2 0 /} [Bl] at 0 0 *5 0 2 /

\put{$y$} at -0.8 11 \put {$z$} at 13 -0.8

 \endpicture
\hspace{2cm}
\end{minipage}
\begin{minipage}{6cm}
\beginpicture

 \setcoordinatesystem units <9pt,9pt> point at 0 0

 \setplotarea x from -1 to 13, y from -1 to 11
 \axis left shiftedto x=0 ticks numbered from 5 to 11 by 5
                                short unlabeled from 1 to 11 by 1 /
 \axis bottom shiftedto y=0 ticks numbered from 5 to 13 by 5
                                short unlabeled from 1 to 13 by 1 /

 \put {\circle*{5}} [Bl] at  6  5
 \put {\circle*{5}} [Bl] at  6  3
 \put {\circle*{5}} [Bl] at  11  1
 \put {\circle*{5}} [Bl] at  11  2
  \put {\circle*{5}} [Bl] at  11  3
   \put {\circle*{5}} [Bl] at  11  0

 \plot 6 11  6 3  11 0  13 0 /

 \vshade 6 3 11 <,z,,> 11 0 11 <,z,,>  13 0 11 /

 \multiput {\multiput {\circle{5}} [Bl] at 0  0 *6 2 0 /} [Bl] at 0 0 *5 0 2 /

\put{$y$} at -0.8 11 \put {$z$} at 13 -0.8

 \endpicture
\end{minipage}
\\[2ex]
 \begin{minipage}{11cm}
 \begin{abbildung} \label{Lemma5} \begin{center}
 $F(y,z)=y^5z+y^3z^3+y^3z^8$ with $\Height(F)=2$ and
 $F^*(y,z)=F(yz+1 \cdot z,z)$ with $\Degreey(F^*)=3$. \end{center} \end{abbildung}
 \end{minipage}
 \end{center}
\end{figure}


\begin{Prop} \label{translationalmove}
Let $f$ be an element of $R/R^p$. Suppose there don't exist subordinate coordinates realizing $\Height(f)$ such that the blowup $\widehat R_a \rightarrow \widehat R'_{a'}$ is monomial. Then 
$$\Revheight_{a'}(f^*) < \Revheight_{a}(f).$$
\end{Prop}

\begin{proof}
Due to remark \ref{Ordery} it is sufficient to show the result for the following situations:

\begin{center}
\begin{tabular}{ccc}
 $F$ & & $F^*$  \\
  \hline
 distant & $\rightarrow$ & distant \\
 distant & $\rightarrow$ & close \\
 distant & $\rightarrow$ & adjacent \\
 close & $\rightarrow$ & adjacent
\end{tabular}
\end{center}
Combining Lemmata 
\ref{degreey} and \ref{lem6} gives
\begin{eqnarray*}
\Revheight(f^*) &\leq& \Height(F^*)-\Bonus(F^*)\\ & \leq & \left(\textnormal{deg}_y(F^*)-2+\Adj(F^*)\right)-\Bonus(F^*) \\ &\leq & \left(\Height(F_d) + \Parity(d)\right)-2+\Adj(F^*)-\Bonus(F^*) \\ 
&\leq& \Height(F) - \left(2-\Adj(F^*)+\Bonus(F^*)-\Parity(d)\right)=:(\triangle).
\end{eqnarray*}
Since by assumption $\varepsilon < \delta$, one can deduce 
that in the four situations described above 
$$(\triangle) \leq \Height(F)-\Bonus(F)$$ holds.
Consider for instance the situation where $F$ is close and $F^*$ is adjacent. In this case $(\triangle)=\Height(F)-(1+\delta-\Parity(F))<\Height(F)-\varepsilon=\Revheight(f)$ as required. \smallskip 
\end{proof}




\subsection*{(H) Horizontal moves} \label{horizontal-move}
$ $ 
\medskip

\noindent The goal of this section is to prove that the \revheight \ does not increase for the horizontal transform $F^*(y,z)=F(yz,z)$. We assume that $(y,z) \in \Ce_{\Fe}$ are chosen in a way such that $\Height(F)=\Height(f)$ and such that the total transform $f^*$ of $f$ under the blowup $\pi$ has expansion $F^*(y,z)=F(yz,z)$ in $\widehat R'_{a'} \cong K[[y,z]]$.  It is obvious that $$\height(F^*) \leq \height(F)$$ (with $\height(F^*)<\height(F)$ if the Newton polygon $N(F)$ of $F$ contains an edge 
whose angle with the horizontal line is bigger than or equal to $45°$).
And, clearly, by moving horizontally the adjacency and hence the \Bonus \ remain the same.
This immediately implies that
\begin{eqnarray*} \Revheight(f^*)&=&\Height(f^*)-\Bonus(f^*) \leq \Height(F^*)-\Bonus(F^*) \\ &\leq& \height(F)-\Bonus(F)=\Revheight(f)
\end{eqnarray*} is fulfilled for all $f \in Q$.  \smallskip


\subsection*{(V) Vertical moves} \label{vertical-move}
$ $ 
\medskip

\noindent In this section it will be shown that under vertical moves elements $f \in Q=R/R^p$ which are not quasi-monomials satisfy $\Revheight(f')<\Revheight(f)$.   Assume that the subordinate coordinates $(y,z) \in \Ce_{\Fe}$ are chosen so that $\Height(F)=\Height(f)$. The total transform $f^*$ of $f$ is given as the equivalence class of the transform $F^*(y,z)=F(y,yz)$ of $F$. As the intricacy is a minimum it suffices to show that
$$\Revheight(F^*) < \Revheight(F).$$ 
\ind Since $F$ is not a monomial we know that $\height(F)>0$, from which a one-line computation   yields 
$$\Height(F^*) \leq \Height(F)-1.$$
\ind By the definition of the bonus it follows that $\Revheight(F^*)< \Revheight(F)$ except possibly if $F$ is adjacent and $\Height(F^*) = \Height(F)-1$. This equality only occurs if $\Width(F)=\deg_z(F)-\Order_z(F)$ equals $1$, i.e., if $F$ is a quasi-monomial. 

\begin{Bem} In the case of width $1$, we may assume, by prior line blowups with center the $z$-axis, that $\Order_z(F)<p$. This combined with $\Width(F)=1$ and $F$ adjacent  implies that $F$ has a pure $y$-monomial $y^m$ with $m\leq p$ (cf. figure 8).  But $m=p$ is not possible because $F$ has its exponents in $\N^2\setminus p\cdot \N^2$, and $m<p$ implies that the order of $f$ (and hence $G$) has dropped below $p$. So quasi-monomials are handled by applying suitable line blowups. We conclude that under vertical moves either the order of $G$ drops or $\Revheight(f')<\Revheight(f)$.
\end{Bem}

	

 \begin{figure}[htbp]
 \begin{center}
 \begin{minipage}{6cm}
 \beginpicture

 \setcoordinatesystem units <7pt,7pt> point at 0 0

 \setplotarea x from -1 to 14, y from -1 to 14
 \axis left shiftedto x=0 ticks numbered from 5 to 14 by 5
                                short unlabeled from 1 to 14 by 1 /
 \axis bottom shiftedto y=0 ticks numbered from 5 to 14 by 5
                                short unlabeled from 1 to 14 by 1 /

 \put {\circle*{5}} [Bl] at  3  3
 \put {\circle*{5}} [Bl] at  4  0

 \put {\circle*{5}} [Bl] at  6  0

 \plot 3 14    3 3  4 0  14 0 /   

 \vshade  3 3 14 <,z,,> 4  0 14  <,z,,>  14 0 14 /

 \put {\large $N(F)$} at  8 8

 \multiput {\multiput {\circle{5}} [Bl] at 0  0 *2 5 0 /} [Bl] at 0 0 *2 0 5 /
 \put {$z$} at 14 -0.8
 \put {$y$} at -0.8 14
 \endpicture
 \hspace{2cm}
\end{minipage}
 \begin{minipage}{6cm}

 \beginpicture

 \setcoordinatesystem units <7pt,7pt> point at 0 0

 \setplotarea x from -1 to 14, y from -1 to 14
 \axis left shiftedto x=0 ticks numbered from 5 to 14 by 5
                                short unlabeled from 1 to 14 by 1 /
 \axis bottom shiftedto y=0 ticks numbered from 5 to 14 by 5
                                short unlabeled from 1 to 14 by 1 /

 \put {\circle*{5}} [Bl] at  3  6
 \put {\circle*{5}} [Bl] at  4  4

 \put {\circle*{5}} [Bl] at  6  6

 \plot 3 14    3 6  4 4  14 4 /   

 \vshade  3 6 14 <,z,,> 4  4 14  <,z,,>  14 4 14 /

 \put {\large $N(F^*)$} at  8 8

 \multiput {\multiput {\circle{5}} [Bl] at 0  0 *2 5 0 /} [Bl] at 0 0 *2 0 5 /
 \put {$z$} at 14 -0.8
 \put {$y$} at -0.8 14
 \endpicture
 \end{minipage}
 \\[2ex]
 \begin{minipage}{12cm}
 \begin{abbildung}\label{ausnahme} \begin{center}
 Configuration where the \revheight  \ increases under blowup. \end{center} \end{abbildung}
 \end{minipage}
 \end{center}
 \end{figure}


\subsection{Decrease of the invariant}  \label{maximieren}  
$ $ \\ 

\noindent In order to prove Theorem \ref{Theo1} it remains, due to Proposition \ref{translationalmove} of paragraph (T)  in section \ref{nonincrease},
 to show that all $f \in Q=R/R^p$ that are not quasi-monomials (which can be resolved directly, see 
paragraph (V) of section \ref{vertical-move})
fulfill
$$(\Revheight(f^*),\Slope(f^*)) <_{lex} (\Revheight(f),\Slope(f)), \ \ \ \ \ (1)$$
where $f^*$ is given as the equivalence class of one of the transforms $F^*(y,z)=F(yz,z)$ or $F^*(y,z)=F(y,yz)$ of  a representative $F(y,z)$ of $f$ with $\Height(F)=\Height(f)$. 

\ind For the purpose of proving $(1)$, fix throughout this section subordinate coordinates $(y,z)$ at $a$ realizing the \height \ of $f$ such that $a' \in E$ is one of the origins of the two charts of the blowup $\pi$. Then the total transform $f^*$ of $f$ under $\pi$ is one of the transforms $F^*(y,z)=F(yz,z)$ or $F^*(y,z)=F(y,yz)$ of $F(y,z)$. 

\ind Due to Proposition \ref{intricacy} of section \ref{nonincrease}, all
elements $f \in Q$ which are not quasi-monomials satisfy  $\Revheight(f^*) \leq  \Revheight(f)$. Hence one is left with the
case that $$\Revheight(f^*)=\Revheight(f). \ \ \ \ \ (2)$$ 
Since the situation $(2)$ doesn't occur when applying translational or vertical moves, it suffices to consider the
horizontal move $F^*(y,z)=F(yz,z)$.
It is obvious that then
 $(2)$ can only happen if the Newton polygon
$N(F)$ of $F$ consists just of edges
whose angle with the horizontal line is smaller than $45°$.
But in this case the vertices of $N(F)$ with the highest
respectively second highest first component transform into
vertices of the Newton polygon $N(F^*)$ of $F^*$ with the same
property. Moreover, then 
$$\Slope(F^*)=\Slope(F)-\alpha_1 < \Slope(F),$$ where $(\alpha_1,\beta_1)$ denotes the vertex of $N(F)$ whose first component has 
the highest value among all vertices of $A$. 
Now assume that $\Slope(f^*)>\Slope(F^*)$. Then
there exists a coordinate change $\varphi'$ which is subordinate to the flag $\Ge$ at $a'$
such that $$\Height(\varphi'(F^*))=\Height(F^*) \textnormal{ and } \Slope(\varphi'(F^*))>\Slope(F^*).$$
One can assume that $\varphi'$ is of the form 
$$\varphi': (y,z) \rightarrow \left(y+A(z),z\right)$$ with $A \in K[[z]]$, $\textnormal{ord}(A) \geq 1$ (if would $A$ depend also on $y$, the respective terms have no effect on the slope).
Let $\varphi$ be the coordinate change subordinate to the flag $\Fe$ at $a$ given by
$$\varphi: (y,z) \rightarrow \left(y + z \cdot A(z),z \right).$$
Then the computation
\begin{eqnarray*}
\varphi' \left(F^*(y,z)\right) &=& \varphi' \left(F(yz,z)\right) \\
&=& F\left((y+A(z))z,z\right) \\
&=& \left(F\left(y+zA(z),z\right)\right)^* \\
&=& \left(\varphi \left(F(y,z)\right)\right)^* 
\end{eqnarray*}
shows that the necessary coordinate change $\varphi'$ at $a'$ stems from the coordinate change $\varphi$ at $a$ and that when applying $\varphi$ and $\varphi'$ the blowup remains monomial. 
In other words, if one realizes $\Slope(f^*)$ after applying the blowup by $\Slope(\varphi'(F^*))$, then $\Slope(\varphi(F))$ automatically realizes $\Slope(f)$. And consequently $\Slope(f^*)<\Slope(f)$. This proves Theorem \ref{Theo1}.\\


\section{Proof of Corollary \ref{Theo2}} \label{termination}
$ $ \\
\noindent Recall that our strategy for improving the singularities of a purely inseparable two-dimen\-sional hypersurface $$G(x,y,z)=x^p+F(y,z)$$ of order equal to the characteristic  is the following:  As long as the equivalence class $f$ of $F$ in $Q=R/R^p$ is not a monomial in a certain point $b=(b_1,a)$  in $\A^3_K$ with $\textnormal{ord}_b(G)=p$, we blow up $\A^3_K$ with center $Z=\{b\}$. Due to Theorem \ref{Theo1} this point blowup  improves the situation (except in the case that $f$ is a quasi-monomial, which can be resolved directly, cf. section \ref{vertical-move}) in the sense that $i_{a'}(f')<i_a(f)$ for all points $a' \in E=\pi^{-1}(Z)$ above $a$, where $i_a(f)=(\Height_a(f)-\Bonus_a(f),\Slope_a(f))$ denotes the local resolution invariant defined in section \ref{Kap1}.  One can hence deduce by induction that point blowups eventually lead to $\Height_a(f)=0$, i.e., that $f$ is monomial.
This is a combinatorial situation: in section \ref{Kap5} it is shown that in this case  the order of the surface can be decreased by finitely many further point and curve blowups.  

\ind To ensure that {\it finitely} many point blowups suffice to transform $f$ in {\it every} point $a \in V(G)$ into a monomial, it will be shown in this section that there are only finitely many closed points $b=(b_1,a)$ on $V(G)$ where $f$ is not monomial in $a$ (and $\textnormal{ord}_b(G)=p$). This establishes the termination of the algorithm described above. 

The result will be proven in two steps: First it is shown -- already for arbitrary dimensional 
purely inseparable hypersurfaces $X$ with order equal to $p$ -- 
that the subset of $X$ containing those points $b$ where the coefficient ideal is not monomial (and the order of $X$ in $b$ is equal to $p$) is {\it Zariski-closed}. Afterwards this result will be used to prove that in the surface case there are only {\it finitely many} such points. \\

\begin{Prop} \label{lemma_1}
Let $G(x,y_1,\ldots,y_n)=x^p+F(y_1,\ldots,y_n)$ with $F \in K[y_1,\ldots,y_n]$ and where $F$ is not a $p$-th power. Denote by $y$ the $n$-tuple of variables $(y_1,\ldots,y_n)$.
Then the set of closed points $b=(0,a_1,\ldots,a_n) \in  \A_K^{1+n}$ such that there exist a local formal coordinate change $\psi$ at $b$ of the form $\psi: (x,y_1,\ldots,y_n) \rightarrow (x-H(y),\alpha_1(y),\ldots,\alpha_n(y))$, where $H \in K[[y]]$ and $\varphi: (y_1,\ldots,y_n) \rightarrow (\alpha_1(y),\ldots,\alpha_n(y))$ is an element of $\textnormal{Aut}(K[[y]])$, and a unit $u \in K[[y]]^*$ with the property that $$G(\psi(x,y)+b)=x^p+u(y) \cdot y^\beta$$ for some vector $\beta \in \N^n \setminus p \cdot \N^n$, is Zariski-open in $\{0\} \times\A_K^n$. 
\end{Prop}

\begin{proof}  
\noindent The assertion of the proposition is clearly equivalent to the statement that the following set is Zariski-open in $\A_K^n$: 
\begin{center} $\textnormal{mon}(F):=\{a \in \A_K^n; \textnormal{there exist} \ \varphi \in \textnormal{Aut}(K[[y]]),\ H \in K[[y]], \ u  \in K[[y]]^* $ \\ $\textnormal{ such that for some }$ $\beta \in \N^n \setminus p \cdot \N^n$: \\ $F(\varphi(y)+a)=u(y) \cdot y^\beta + H(y)^p\}.$  
\end{center}  
Note that if a series $A \in K[[y]]$ factors into a monomial times a unit $U \in K[[y]]^*$, i.e., $$A(y)=U(y) \cdot y^\gamma,$$ where at least one of the components of $\gamma$ is not a multiple of the characteristic $p$ of the ground field $K$, then there exists a coordinate change $\tau \in \textnormal{Aut}(K[[y]])$ such that $$A(\tau(y))=y^\gamma.$$ This is due to the fact that a unit $U \in K[[y]]^*$ has a $r$-th root $U^{1/r}$ in $K[[y]]^*$ if $(r,p)=1$ (and can for example be deduced from Lemma 4.2 in \cite{Brown}).
Since the image of a $p$-th power under an automorphism $\tau \in \textnormal{Aut}(K[[y]])$ is again a $p$-th power, the set $\textnormal{mon}(F)$ can be rewritten as
\begin{center} $\textnormal{mon}(F)=\{a \in \A_K^n; \textnormal{there exist}\ \varphi \in \textnormal{Aut}(K[[y]]), \ H \in K[[y]]  $ \\ $\textnormal{ such that for some }$ $\beta \in \N^n \setminus p \cdot \N^n$: \\ $F(\varphi(y)+a)=y^\beta + H(y)^p\}.$  
\end{center}  
We will prove that this set is Zariski-open in $\A_K^n$ by following a construction which will be explained in detail in the forthcoming article  \cite{BW}: 
Consider for a fixed point $a \in \A^n$ the equation 
$$F(\varphi(y)+a)= y^\beta + H(y)^p. \hspace{10pt} (\star)$$
By Artin's Approximation Theorem \cite{MR0268188} it follows that if for some vector $\beta \in \N^n \setminus p \cdot \N^n$ there exist solutions $\overline{\varphi}(y)=(\overline{\alpha_1}(y),\ldots,\overline{\alpha_n}(y))$ and $\overline{H}(y)$ of $(\star)$
in the ring $K[[y]]$ of formal power series, then there already exist  solutions ${\varphi}(y)=(\alpha_1(y),\ldots,\alpha_n(y))$ and ${H}(y)$  of $(\star)$ in the henselisation of $K[y]$, i.e., in the ring $K\langle \langle y \rangle \rangle$ of algebraic power series in $n$ variables, such that both solutions agree modulo $\langle y \rangle^c$ for a chosen constant $c \in \N$. Note that if one chooses $c=2$, then the property for $\overline{\varphi}$ to be an automorphism is also ensured for $\varphi$. 
Since $H$ and the components $\alpha_i$ of $\varphi$ are elements of $K\langle \langle y \rangle \rangle$, they are regular functions on an \'etale neighborhood $\theta_a: (V,v) \rightarrow (\A_K^n,a)$ of $a=\theta_a(v)$. Now consider the monomial locus $\textnormal{mon}(Q,a)$ of $$Q(y):=F(\varphi(y)+a)-H(y)^p$$ 
 in $V$, i.e., the set of points $v' \in V$ such that there exist local coordinates $w=(w_1,\ldots,w_n)$ at $v'$ with $Q(w+v')=w^\gamma$ in $\widehat{\mathcal{O}}_{V,v'} = K[[w]]$ for some $\gamma \in \N^n$. 
In \cite{BW} it is proven that $\textnormal{mon}(Q,a)$ is a Zariski-open subset of $V$. Due to $\widehat{\mathcal{O}}_{V,v'} = \widehat{\mathcal{O}}_{\A_K^n,\theta_a(v')}$, $v' \in \textnormal{mon}(Q,a)$ implies that $F(w+\theta_a(v'))=w^\gamma+H(w+\theta_a(v'))^p$.
Note that at first sight it seems to be possible that $\gamma \in p \cdot \N^n$, and in this case $\theta_a(v')$ wouldn't be an element of $\textnormal{mon}(F)$. But if all components of $\gamma$ are multiples of $p$ then $F(w+\theta_a(v'))=w^\gamma+H(w+\theta_a(v'))^p$ would be a $p$-th power, which contradicts our assumption (since $F(w) \in K[w]$ is a $p$-th power if and only if $F(\phi(w)+c)$ is for all $\phi \in \textnormal{Aut}(K[[w]])$ and all $c \in \A_K^n$). Consequently $\theta_a(v')$ is contained in $\textnormal{mon}(F)$.
By the openness of \'etale morphisms it follows that $\theta_a(\textnormal{mon}(Q,a))$ is an open subset of $\textnormal{mon}(F)$.  

\ind This procedure can be carried out for all
points $a \ \in \ \textnormal{mon}(F)$. Then the set
$$\bigcup_{a \in \textnormal{mon}(F)}  \theta_a(\textnormal{mon}(Q,a))$$ clearly  equals $\textnormal{mon}(F)$ and is as a union of Zariski-open sets itself Zariski-open.  \\
\end{proof}


\begin{Prop} \label{lemma_2}
Let $f$ be an element of $R$ which is not a $p$-th power. Then the closed points $a \in V(f) \subset \A^2_K$ in which $f$ has order $\textnormal{ord}_a(f) \geq p$ and in which $f$ is, when considered as an element of $R/R^p$, not monomial, are isolated (in particular, finite in number).
\end{Prop}

\begin{proof} 
\noindent Note that the set of closed points $a \in \A_K^2$ in which $f \in R/R^p \setminus \{0\}$ is monomial, is equal to the set $\textnormal{mon}(F)$ (with $n=2$) introduced in the proof of the last theorem, which was shown to be Zariski-open. Its complement in $\A_K^2$ -- which equals the set of points of $\A_K^2$ in which $f$ is not monomial -- is hence Zariski-closed. We are only interested in those points $a \in \A_K^2 \setminus F_{mon}$ in which the order of $f \in R$ is bigger or equal to $p$ (which clearly implies that $a \in V(f)$), thus in the points of the intersection 
$$(\triangle):=\left(\A_K^2 \setminus F_{mon}\right) \ \cap \ \{a \in \A_K^2; \textnormal{ord}_a(f) \geq p\}.$$
By the upper-semicontinuity of the order function it is clear that also the second of these two sets is a Zariski-closed subset of $\A_K^2$. Consequently, the points $a$ in $(\triangle)$ form an algebraic subset of $\A^2$. Moreover, the set $(\triangle)$ is a subset of the singular locus $\textnormal{Sing}(X)$ of $X=V(f) \subset \A^2$. And since any algebraic curve has only finitely many singular points, the set $(\triangle)$ consists of at most finitely many points. \\
\end{proof}


\section{Monomial Case} \label{Kap5} 
$ $ \\
\noindent The goal of this section is to decrease the order of the purely inseparable equation $$G=x^p+F(y,z)$$ with $\textnormal{ord}_0(G)=p$ in every point of the singular surface $X=V(G)\subset \A^3_K$ by a finite sequence of blowups to a value which is smaller than $p$. 
In section \ref{Kap2}, especially in remark \ref{Bem4}, we explained why a point blowup of such a surface  can be reduced to a point blowup of the plane curve $F(y,z)=0$ modulo $p$-th powers. 
Moreover, in section \ref{Kap3} it was shown that a finite number of point blowups transforms $F$ in every point $b$ of $X$ with $\textnormal{ord}_b(G)=p$ into a monomial times a unit (or makes the order of $G$ drop).
This is done by using a local resolution invariant associated to $F$. To decrease the order of $G$  one can therefore assume that $G$ is of form $$G(x,y,z)=x^p + y^mz^n A(y,z)$$ 
with $(m,n) \in \N^2 \setminus p \cdot \N^2$, $m+n \geq p$ and $A(0,0) \neq 0$. After a formal coordinate change one can furthermore assume that $A(y,z)=1$ (for a detailed argumentation of this, see the proof of Lemma \ref{lemma_1} in section \ref{Kap2}).
Once $F$ is monomial, there is an immediate combinatorial way to lower the order of $G$, which will be described in the sequel (this is a classical argument which works in any dimension).

Let $(y,z)$ and $(x,y,z)$ be regular parameter systems of $\widehat R_a$ and $\widehat S_{b}$, where $\widehat R_a$ and
$\widehat S_{b}$ denote the
completion of the localization of the coordinate ring $R$ of $\A_K^2$ at the point $a$ respectively the coordinate ring $S$ of $\A_K^3$ at $b=(b_1,a)$. Furthermore let $F(y,z)$
and $G(x,y,z)$ be the expansions of $f \in R/R^p$ and $g
\in S$ with respect to the chosen local coordinates.

The center of the next blowup is defined by means of the top locus $\hbox{top}(G)$ of $X$. Recall that $\hbox{top}(G)$ consists of those points $b \in X$ where the local order of $G$ attains its maximal value. Thus 
$$\hbox{top}(G)=\{b \in X; \hbox{ord}_b(G)=p\}.$$
We may assume that the top locus has no self intersections (otherwise further point blowups have to be applied to ensure this condition). \\

\noindent Then there are three different cases according to the values of $m$ and $n$:
 \\

\noindent {\it (1) Case $m \geq p$:} This implies that $G \in
\langle x,y \rangle^p$ and hence the $z$-axis is included in the top locus of
$V(G)$.
In this case we choose locally the $z$-axis as the center of the blowup. This
yields in the $x$-chart a variety which is smooth in all of its points and in the $y$-chart
$G^*(x,y,z)=y^p \cdot (x^p+y^{m-p}z^n)$
with $m-p < m$. Hence induction can be applied until $m<p$. \\
\\
{\it (2) Case $n \geq p$:} Symmetrically, we choose locally the
$y$-axis as center and apply induction
until $n <p$. \\ \\Iterate this process until both $m$ and $n$ are less than $p$. \\
\\ 
{\it (3) Case $m<p$ and $n<p$:} In this situation we choose as
center the origin of $\A_K^3$, which is in this case the only element of the top locus of $V(G)$. This yields in the $x$-chart a variety which is smooth in all of its points. In the $y$-chart, and analogously in the $z$-chart, one
gets $G^*(x,y,z)=y^p(x^p+y^{m+n-p}z^n)$ with $m+n-p <m$, and
therefore induction on $(m,n)$ works.
\\	

\ind Altogether this yields that $G$ is given,  after finitely many blowups where the centers have to be chosen in the manner described above, locally in every (singular) point of $V(G)$ by $$G(x,y,z)=x^p+F(y,z)$$ with $\textnormal{ord}(F)<p$. \\

\begin{Bem} \label{transversality} 
In order to achieve an embedded resolution of the purely inseparable two-dimensional hypersurface $X$, it is necessary that in every step of the resolution algorithm the chosen center is transversal to the already existing exceptional divisor. 
In this section it was shown that the only higher dimensional centers which are possibly required during our algorithm, are the $y$- and the $z$-axis of $\A^3_K$. If the already existing exceptional divisor is not yet transversal to one of the chosen axis, then one first has to apply point blowups in order to achieve transversality.
\end{Bem}


\section{A second resolution invariant}  \label{invariant2}

\noindent In this section we will define a second local resolution invariant which also works for surfaces in characteristic $p$. It is a modification of the classical resolution invariant used in characteristic zero. Furthermore we will prove that this invariant also drops lexicographically under point blowups (except for a specific quasi-monomial, which can be resolved directly) and hence can be used alternatively to prove Theorem \ref{Theo1} and Corollary \ref{Theo2}. \\

\subsection{Definition of the second invariant} \label{def2}
$ $ \\
\\
Let $R$ be the coordinate ring of the affine plane $\A^2_K$ over an algebraically closed field $K$ of characteristic $p$. Furthermore let $R_a$ be the localization of
$R$ at a closed point $a$ of $\A_K^2$ and $\widehat R_a$ its completion. 
We fix for the entire section a local flag $\Fe$ in $\A_K^2$ at $a$. By $(y,z)$ we denote local coordinates subordinate to $\Fe$ and by $F=F(y,z)$ the expansion of an element $f \in Q=R/R^p$ in $K[[y,z]]$. Moreover let $N=N(F)$ be the Newton polygon of $F$ and $A \subset \N^2$ its set of vertices.  

Denote by
$\Orderz(F) =
\min_{(\alpha,\beta) \in A} \ \beta$
the {\it order of $F$   with respect to $z$} (see figure \ref{Definition divorder}). Then the {\it \dorder}  of $F$ is defined as
$$\Dorder(F)=\textnormal{ord}(F)-\Ordery(F)-\Orderz(F).$$
It is thus the maximal side length of all equilateral axes-parallel triangles which can be inscribed in $((\Ordery(F),\Orderz(F))+\R_+^2) \setminus N(F)$ (see figure \ref{Definition divorder}). Or in other words, if $y^m z^n$ is the maximal monomial which can be factored from $F(y,z)$ and $H(y,z)=y^{-m}z^{-n}F(y,z)$, then $\Dorder(F)=\textnormal{ord}(H).$ The \dorder \ thus measures the distance of $F$ from being a monomial up to units. It will constitute together with a correction term the first component of our new resolution invariant. \\


 \begin{figure}[h]
 \begin{center}
 \begin{minipage}{6cm}
 \beginpicture

 \setcoordinatesystem units <6pt,6pt> point at 0 0

 \setplotarea x from -1 to 24, y from -1 to 18
 \axis left shiftedto x=0 ticks numbered from 25 to 18 by 25
                                short unlabeled from 1 to 18 by 1 /
 \axis bottom shiftedto y=0 ticks numbered from 25 to 24 by 25
                                short unlabeled from 1 to 24 by 1 /

 \put {\circle*{5}} [Bl] at  6  15
 \put {\circle*{5}} [Bl] at  12  6
 \put {\circle*{5}} [Bl] at  22  3

 \plot 6 18  6 15 12 6  22 3  24 3  /   

 \vshade   6 15 18  <,z,,>  12 6 18 <,z,,> 22 3 18 <,z,,> 24 3 18 /

\put {\vector(1,0){36}} at 3 15.7

\put {\vector(-1,0){36}} at 3 15

\put {\small{$\Orderz(F)$}} at 3 15.7




\put {\vector(0,1){18}} at 22  1.5

\put {\vector(0,-1){18}} at 22  1.5

\put {\small{$\Ordery(F)$}} at 19.4 1.5




\put {\small{$\Dorder(F)$}} at 10.2 2.1

\setdots <1 pt>

\plot 6 3   15 3  /

\plot 6 3   6 12   15 3  /

\setdots <6 pt>

\plot 6 3   6 15  /

\plot 6 3   22 3  /


 \put {$z$} at 24 -0.8
 \put {$y$} at -0.8 18
 \endpicture

\end{minipage}
  \\[2ex]
 \begin{minipage}{11cm}
 \begin{abbildung}\label{Definition divorder} \begin{center} $\Ordery(F)$, $\Orderz(F)$
 and $\Dorder(F)$ of
 $F$. \end{center} \end{abbildung}
 \end{minipage}
 \end{center}
 \end{figure}

\noindent The second component of our new resolution invariant will be defined as follows: If $N$ is not a quadrant, we set the  {\it \dent} of $F$ as the vector
$$\Dent(F)=(\alpha_1-\alpha_2,\beta_2-\beta_1),$$ where $(\alpha_1,\beta_1)$ and
$(\alpha_2,\beta_2)$ denote those elements of $A$ whose first
component have the highest respectively second highest value among
all vertices of $A$.  The first respectively second component of this vector will be denoted by $\Updent(F)$ and $\Indent(F)$ and called the {\it \updent} respectively {\it \indt} of $F$. 
 
\ind It is clear that $\Height(F_d)=\Dorder(F_d)$. Therefore 
Lemma \ref{lem6} of section \ref{Kap3.1} tells us
that also the \dorder \ can increase in characteristic $p > 0$ under blowup at most by $1$. But the modification of the measure $\Dorder(F)$ in order to get  a decreasing resolution invariant is more involved. Recall that the {\it adjacency} $\Adj(F)$ of $F$ is equal to $2$,
$1$ or $0$ according to $F$ being {\it adjacent}, $\Ordery(F)=0$,
{\it close}, $\Ordery(F)=1$, or {\it distant}, $\Ordery(F)\geq 2$. The {\it \defect} of $F$ is defined as follows: 

If $\Dorder(F)=\Degreey(F)-\Ordery(F)$, the {\it \defect}of $F$ is defined to be $1+\delta$ for $F$ being adjacent, $\varepsilon$ for $F$ being close and $0$ otherwise.
If $\Dorder(F)=\Degreey(F)-\Ordery(f)-1$, the {\it \defect}of $F$ is set equal to $\delta$ for $F$ being adjacent and $0$ otherwise. And if $\Dorder(F) \leq \Degreey(F)-\Ordery(f)-2$, the {\it \defect}of $F$ is defined as $0$.

\ind In all cases $\varepsilon, \delta$ denote arbitrarily chosen  positive numbers between $0$ and $1$ with $\varepsilon<\delta$. 

\ind The \defect is a correction term that takes into account -- additionally to the position of the Newton polygon $N(F)$ with respect to the $z$-axis  -- also the occurrence of edges in $N(F)$ whose angle with the horizontal line is bigger than $45°$. This is similar to the correction term \Bonus \ defined earlier.
Note that the definition breaks the symmetry
between $y$ and $z$. In figure \ref{Bonus2} some possible configurations of $N(F)$ and the corresponding values of $\Defect(F)$ are illustrated. \\
      
       

 \begin{figure}[h]
 \begin{center}
 \begin{minipage}{4cm}
 \beginpicture

 \setcoordinatesystem units <8pt,8pt> point at 0 0

 \setplotarea x from -1 to 9, y from -1 to 6
 \axis left shiftedto x=0 ticks numbered from 5 to 6 by 5
                                short unlabeled from 1 to 6 by 1 /
 \axis bottom shiftedto y=0 ticks numbered from 5 to 9 by 5
                                short unlabeled from 1 to 9 by 1 /

 \put {\circle*{5}} [Bl] at  1  5
 \put {\circle*{5}} [Bl] at  3  2
 \put {\circle*{5}} [Bl] at  7  0

 \plot 1 6    1 5   3 2   7 0   9 0   /   

 \vshade   1 5 6  <,z,,> 3 2 6 <,z,,> 7 0 6 <,z,,> 9 0 6 /

\put{\small{$\Defect(F)=\delta$:}} at 4.5 7.5

 \put {$z$} at 9 -0.8
 \put {$y$} at -0.8 6
 \endpicture
\hspace{1cm}
\end{minipage}
\begin{minipage}{4cm}
 \beginpicture

 \setcoordinatesystem units <8pt,8pt> point at 0 0

 \setplotarea x from -1 to 9, y from -1 to 6
 \axis left shiftedto x=0 ticks numbered from 5 to 6 by 5
                                short unlabeled from 1 to 6 by 1 /
 \axis bottom shiftedto y=0 ticks numbered from 5 to 9 by 5
                                short unlabeled from 1 to 9 by 1 /

 \put {\circle*{5}} [Bl] at  1  5
 \put {\circle*{5}} [Bl] at  2  2
 \put {\circle*{5}} [Bl] at  6  0

 \plot 1 6    1 5   2 2   6 0   9 0   /   

 \vshade   1 5 6  <,z,,> 2 2 6 <,z,,> 6 0 6 <,z,,> 9 0 6 /

\put{\small{$\Defect(F)=0$:}} at 4.5 7.5

 \put {$z$} at 9 -0.8
 \put {$y$} at -0.8 6
 \endpicture
\hspace{1cm}
\end{minipage}
\begin{minipage}{4cm}
 \beginpicture

 \setcoordinatesystem units <8pt,8pt> point at 0 0

 \setplotarea x from -1 to 9, y from -1 to 6
 \axis left shiftedto x=0 ticks numbered from 5 to 6 by 5
                                short unlabeled from 1 to 6 by 1 /
 \axis bottom shiftedto y=0 ticks numbered from 5 to 9 by 5
                                short unlabeled from 1 to 9 by 1 /

 \put {\circle*{5}} [Bl] at  0 4
 \put {\circle*{5}} [Bl] at  2  2
 \put {\circle*{5}} [Bl] at  6  1

 \plot 0 6    0 4  2 2   6 1   9 1   /   

 \vshade   0 4 6  <,z,,> 2 2 6 <,z,,> 6 1 6 <,z,,> 9 1 6 /

\put{\small{$\Defect(F)=\varepsilon$:}} at 4.5 7.5

 \put {$z$} at 9 -0.8
 \put {$y$} at -0.8 6
\endpicture
\end{minipage}
    \\[2ex]
 \begin{minipage}{11cm}
 \begin{abbildung} \label{Bonus2} \begin{center} Some examples for $\Defect(F)$. \end{center}\end{abbildung}
 \end{minipage}
 \end{center}
 \end{figure}


\noindent Now these measures will be associate in a coordinate independent
way to residue classes $f$ in $R/R^p$. Denote by $\Ce=\Ce_\Fe$  as usual the set of subordinate local coordinates $(y,z)$
in $\widehat R_a$. Since the highest vertex $c=(\alpha,\beta)$
of $N=N(F)$ does not depend on the choice of the subordinate
coordinates, $\Degreey(F)$ and
$\Orderz(F)$ take the same value for all elements in $\Ce$. Recall that also the value $\textnormal{ord}(F)$ is independent of the choice of subordinate coordinates and is called the {\it order} of $f \in R/R^p$.

For $f\in R/R^p$ with expansion $F=F(y,z)$ at $a$ with respect to $(y,z)\in
\Ce$ we set
\begin{eqnarray*} \Dorder_a(f) &=& \min  \{\Dorder(F); (y,z)\in
\Ce\} \\
&=& \textnormal{ord}(F)- \textnormal{ord}_z(F) - \max\{\textnormal{ord}_y(F); (y,z) \in \Ce\}
\end{eqnarray*}
and call
it the {\it \dorder} of $f$. This number only depends on $f$, the
point $a$ and the chosen flag $\Fe$. 

\ind We say that $f$ is {\it
monomial} at $a$ if there exists a local (not necessarily subordinate) coordinate change 
transforming $F$ into a monomial $y^\alpha z^\beta$ times a unit in $K[[y, z]]$. Note 
that this is in particular the case if  $\Dorder_a(f)=0$ (which is equivalent to $\Height_a(f)=0$).  

\ind Since $\Adj(F)$ takes the same value, say $\Adj_a(f)$, for all coordinates realizing $\Dorder(f)$, it is a simple matter to check that also $\Defect(F)$ takes the same value, say $\Defect_a(f)$, for all these coordinates.
Therefore the {\it \revdorder}
\begin{eqnarray*} \Revdorder_a(f)&:=&
\Dorder_a(f)-\Defect_a(f) \\ &=&\min\{\Dorder(F)-\Defect(F); (y,z)\in
\Ce\}\end{eqnarray*} only depends on $f\in R/R^p$, the point $a$ and the chosen
flag $\Fe$. This will be the first component of our new local resolution invariant.
We will leave out the reference to the point $a$ when $a$ is fixed and simply write $$\Revdorder(f)=\Dorder(f)-\Defect(f).$$ \smallskip


\ind The second component of our new local resolution invariant will be
$$\Dent_a(f):=(\Updent_a(f),\Indent_a(f)),$$ where $\Updent(F)$  is minimized and afterwards $\Indent(F)$  is maximized over all subordinate coordinates $(y,z) \in \Ce$ for which the expansion $F(y,z)$ fulfills $\Dorder(F)=\Dorder_a(f)$.
It also only depends on $f\in R/R^p$, the point $a$ and the chosen flag
$\Fe$. Again we omit the reference to $a$ and simply write
$\Dent(f)$. 

\ind The new local resolution invariant of $f\in R/R^p$ at $a$ with respect to $\Fe$ is then defined as $$j_a(f) = (\Revdorder_a(f), \Dent_a(f)),$$  considered
with respect to the lexicographic order with $(0,1) < (1,0)$. Note that $j_a(f)$ is an element of a well-ordered set. \\


\subsection{Non-increase of the \revdorder  under blowup}
$ $ \\
\\
In order to prove Theorem \ref{Theo1} for the resolution invariant defined in section \ref{def2}, we start by showing the following proposition:

\begin{Prop} \label{complicacy}
Let $f$ be an element of $R/R^p$, which is not a (specific) quasi-monomial, and let $F \in K[[y,z]]$ be its expansion with respect to subordinate coordinates $(y,z) \in \Ce_{\Fe}$ realizing the \dorder \ of $f$. Furthermore let $F^*(y,z)$ be one of the transformations $F^*(y,z)=F(yz+tz,z)$ or $F^*(y,z)=F(y,yz)$ and $f^*$ the corresponding element in $R'/R'^p$. Then 
$$\Revdorder(f^*) \leq \Revdorder(f). \ \ \ \ \ (1)$$ 
Moreover, if either the translational move (T) is forced, or there exist subordinate coordinates realizing the height of $f$ such that the 
blowup $\widehat {R}_a \rightarrow \widehat R'_{a'}$ is given by move (V), then 
$$\Revdorder(f^*)  < \Revdorder(f),$$ where $F^*(y,z)=F(yz+tz,z)$ with $t \neq 0$.
\end{Prop}

\ind The proposition above will again be proven separately for the three different moves (T), (H) and (V) defined in section \ref{Kap2}. \medskip


\subsection*{(T) Translational moves} \label{decrease2}
$ $ \medskip

\noindent Assume
that there don't exist subordinate coordinates at $a$ realizing the \dorder \ of $f$ such that the blowup is monomial. In this situation the total transform 
$f^*$ of $f$ under the blowup $\pi$ is given as the equivalence class of the total transform 
$F^*(y, z ) = F (yz + tz , z )$ where $t \in K^*$, of a representative $F (y, z )$ of $f$ with 
$\Dorder(F ) = \Dorder(f )$. Fix such minimizing subordinate coordinates $(y, z ) 
\in \Ce_{\Fe}$
and denote by $F (y, z )$ in the sequel always the expansion of $f$ with respect to these 
chosen coordinates. 

\ind Denote by $d$ the order of $f$ and by $f_d$ its initial form. The {\it parity} $\Parity(d)$ of $d$ is defined as in section \ref{Kap3.1}, i.e., set equal to $1$ if $d \equiv 0 \textnormal{ mod } p$, and $0$ otherwise.  

Since $\Height(F_d)=\Dorder(F_d)$, Lemmata \ref{lem5} and \ref{lem6} of section \ref{Kap3.1} can be immediately applied to the \dorder \ of $F_d$ respectively $f_d$. One hence gets $$\Degreey(F^*) \leq \Dorder(F_d)+\Parity(d).$$ 
\begin{Bem} Note that the above inequality nevertheless only implies a possible increase of the \dorder \ if the Newton polygon $N(F^*)$ of $F^*$
consists only of edges whose angle with the horizontal line is smaller or equal than $45°$, i.e., if $\Height(F^*)=\Dorder(F^*)$ (see figures  \ref{Lemma5} and \ref{Lemma5b}). And moreover, it for sure 
decreases in the case that $N(F^*)$ contains edges with slope smaller than $-2$, i.e., if $\Height(F^*)-\Dorder(F^*) > 1$. \\
       
      

\begin{figure}[h]
\begin{center}
\begin{minipage}{6cm}
\beginpicture

 \setcoordinatesystem units <9pt,9pt> point at 0 0

 \setplotarea x from -1 to 13, y from -1 to 11
 \axis left shiftedto x=0 ticks numbered from 5 to 11 by 5
                                short unlabeled from 1 to 11 by 1 /
 \axis bottom shiftedto y=0 ticks numbered from 5 to 13 by 5
                                short unlabeled from 1 to 13 by 1 /

 \put {\circle*{5}} [Bl] at  1  5
 \put {\circle*{5}} [Bl] at  3  3
 \put {\circle*{5}} [Bl] at  4  3

 \plot 1 11  1 5  3 3  13 3 /

 \vshade 1 5 11 <,z,,> 3 3 11 <,z,,>  13 3 11 /

 \multiput {\multiput {\circle{5}} [Bl] at 0  0 *6 2 0 /} [Bl] at 0 0 *5 0 2 /

\put{$y$} at -0.8 11 \put {$z$} at 13 -0.8

 \endpicture
\hspace{2cm}
\end{minipage}
\begin{minipage}{6cm}
\beginpicture

 \setcoordinatesystem units <9pt,9pt> point at 0 0

 \setplotarea x from -1 to 13, y from -1 to 11
 \axis left shiftedto x=0 ticks numbered from 5 to 11 by 5
                                short unlabeled from 1 to 11 by 1 /
 \axis bottom shiftedto y=0 ticks numbered from 5 to 13 by 5
                                short unlabeled from 1 to 13 by 1 /

 \put {\circle*{5}} [Bl] at  6  5
 \put {\circle*{5}} [Bl] at  6  3
 \put {\circle*{5}} [Bl] at  7  1
 \put {\circle*{5}} [Bl] at  7  2
  \put {\circle*{5}} [Bl] at  7  3
   \put {\circle*{5}} [Bl] at  7  0

 \plot 6 11  6 3  7 0  13 0 /

 \vshade 6 3 11 <,z,,> 7 0 11 <,z,,>  13 0 11 /

 \multiput {\multiput {\circle{5}} [Bl] at 0  0 *6 2 0 /} [Bl] at 0 0 *5 0 2 /

\put{$y$} at -0.8 11 \put {$z$} at 13 -0.8

 \endpicture
\end{minipage}
\\[2ex]
 \begin{minipage}{11cm}
 \begin{abbildung} \label{Lemma5b} \begin{center}
 $F(y,z)$ with $\Dorder(F)=2$ and
 $F^*(y,z)=F(yz+1 \cdot z,z)$ with $\Degreey(F^*)=3$, but $\Dorder(F^*)=1<2=\Dorder(F)$. \end{center} \end{abbildung}
 \end{minipage}
 \end{center}
\end{figure}

\end{Bem}

\noindent Due to remark \ref{Ordery} of section \ref{Kap3.1} it follows that it remains to consider the following situations 

\begin{center}
\begin{tabular}{ccc}
 $F$ & & $F^*$  \\
  \hline
 distant & $\rightarrow$ & distant \\
 distant & $\rightarrow$ & close \\
 distant & $\rightarrow$ & adjacent \\
 close & $\rightarrow$ & adjacent
\end{tabular}
\end{center}
Investigating these four cases in detail, one can show similarly as in the proof of Proposition \ref{translationalmove} in section \ref{Kap3.1} that 
$$\Revdorder(f^*)<\Revdorder(f).$$


\subsection*{(H, V) Horizontal and vertical moves} \label{moves2}
$ $ \medskip

\noindent The goal of this section is to prove Proposition \ref{complicacy} for the two monomial transformations $F^*(y,z)=F(yz,z)$ and $F^*(y,z)=F(y,yz)$. 

Assume for this purpose throughout this section that $(y,z)$ are subordinate coordinates realizing $\Dorder(F)=\Dorder(f)$ such that the total transform $f^*$ of $f$ under the blowup $\widehat R_a \rightarrow R'_{a'}$ is given as the equivalence class of one of the transforms $F^*(y,z)=F(yz,z)$ respectively $F^*(y,z)=F(y,yz)$ of $F$.

\ind In section 
\ref{vertical-move} we already proved the analogous statement for the measure \revheight \ defined in section \ref{Kap1}. And since the argumentation runs here quite similar, we will skip some computational parts of the proof of Proposition \ref{complicacy}. 

\ind First note that for both, the horizontal and the vertical move, the inequality $\Dorder(F^*) \leq \Dorder(F)$ holds for all series $F \in K[[y,z]]$.
We start by establishing Proposition \ref{complicacy} for the horizontal move. It is not too hard to check that if $N(F)$ contains at least one edge whose angle with the horizontal line is bigger than $45°$, i.e., if 
$(\Degreey(F)-\Ordery(F))-\Dorder(F) \geq 1$, then $\Defect(F) \in \{0,\const\}$ and $\Dorder(F^*) \leq \Dorder(F)-1$. This immediately implies $$\Dorder(F^*)-\Defect(F^*)<\Dorder(f)-\Defect(f)=\Revdorder(f).$$ We are hence left with series $F$ whose Newton polygon consists only of edges 
whose angles with the horizontal line are smaller or equal than $45°$. 
Some further, but easy, considerations show that in this case the inequality $$ \Dorder(F^*)-\Defect(F^*) \leq \Dorder(f)-\Defect(f)=\Revdorder(f)$$ is always fulfilled. And since all coordinate changes subordinate to the flag $\Fe$ leave the highest vertex of $N(F^*)$ fixed,  inequality $(1)$ follows. 
 
Now we will turn to the vertical move $F^*(y,z)=F(y,yz)$. We will assume that $N(F)$ contains at least one edge 
whose angle with the horizontal line is bigger than $45°$ (otherwise $N(F^*)$ is already a quadrant). 
It can be seen easily that then $\Defect(F)$ is either $0$ or $\const$. In the case that $\Defect(F)=0$, the inequality $(1)$ follows immediately. Therefore, let $\Defect(F)=1$. This implies that $F$ is adjacent and $\Dorder(F)=(\Degreey(F)-\Ordery(F))-1$. 
Furthermore it is a simple matter to check that if $N(F)$ contains an edge whose angle with the horizontal line is smaller or equal than $45°$, then $\Dorder(F^*) \leq \Dorder(F)-1$, hence no increase of the \revdorder can happen. So we are left with the case that $N(F)$ contains only edges whose angles with the horizontal line are bigger than $45°$. 
It is a simple matter to check that then an increase of the \revdorder is only possible if $F$ is of the form $$F(y,z)=z^m \cdot (c  y^2 + d  z) + H(y,z)$$ with $m \in \N$, $c, d \in K^*$ and $H \in K[[y,z]]$ with $N(H) \subset N(F) \setminus A$. Obviously this series is a special quasi-monomial (see section \ref{vertical-move}) and hence can be transformed into a monomial  times a unit by a finite number of further blowups, indeed here only one further blowup (and possibly a subsequent coordinate change) is necessary.  

Together with the investigation of translational moves in the last section this proves Proposition \ref{complicacy}. \\


\subsection{Decrease of the invariant}
$ $ \\
\\
In order to show that the invariant $j(f)=(\Revdorder(f),\Dent(f))$ decreases for all $f \in Q=R/R^p$ which are not quasi-monomials, it remains due to Proposition \ref{complicacy} to prove the inequality 
$$(\Revdorder(f^*),\Dent(f^*)) <_{lex} (\Revdorder(f),\Dent(f)),$$ where $f^*$ corresponds to one of the transforms $F^*(y,z)=F(yz,z)$ or $F^*(y,z)=F(y,yz)$ of a representative $F(y,z)$ of $f$ with $\Dorder(F)=\Dorder(f)$ and $\Updent(F)=\Updent(f)$, in the case that $$\Revdorder(f^*)=\Revdorder(f)\ \ \ \ \ (\triangle).$$

\ind Fix for this purpose  subordinate coordinates $(y,z) \in \Ce_{\Fe}$ with $\Dorder(F)=\Dorder(f)$ and $\Updent(F)=\Updent(f)$ such that the transform $f^*$ of $f$ is given as the equivalence class of one of the series $F^*(y,z)=F(yz,z)$ or $F^*(y,z)=F(y,yz)$.  
 
We will first concentrate on the
horizontal transformation $F^*(y,z)=F(yz,z)$. In this case one can show similarly as in section \ref{maximieren} that under the assumption $(\triangle)$ the inequality $$\Dent(f^*)<\Dent(f)$$ holds.
Now we will treat the vertical move $F^*(y,z)=F(y,yz)$. It is easy to see that $(\triangle)$ can only occur if the Newton polygon $N(F)$ consists just of edges whose angle with the horizontal line is bigger than $45°$. But in this case the vertices of $N(F)$ with the highest
respectively second highest first component transform into
vertices of the Newton polygon $N(F^*)$ of $F^*$ with the same
property. Hence it follows easily that $\Updent(f^*) \leq \Updent(F^*) < \Updent(F)$.   \\


\section{Alternative approach for surface resolution in positive characteristic} \label{modified_proof}
$ $ \\
\noindent In this section we will indicate an alternative approach for resolution of surfaces which are defined by purely inseparable equations over an algebraically closed field $K$ of positive characteristic. It is based on a theorem which characterizes in any dimension completely the shape of the initial form of those purely inseparable polynomials for which the \dorder \ increases under a translational blowup (see Thm. 1, sec. 5 and Thm. 2, sec. 12 in \cite{Hausera}). We will briefly recall the theorem without giving its proof: 

\begin{Theo}
Let $\pi: (W,q') \rightarrow (W,q)$ be a local point blowup of $W=\A^{1+m}$ with center $Z=\{q\}$ the origin. Let $(x,w_m,\ldots,w_1)$ be local coordinates at $q$ such that $$G(x,w)=x^p+w^r \cdot \hat{F}(w) \in \widehat{\mathcal{O}}_{W,q}$$ has order $p$ and $\Dorder_q(w^r \cdot \hat{F})=\textnormal{ord}_q(\hat{F})$ at $q$ with exceptional divisor $w^r=0$. Let $G'$ and $F'$ be the strict transforms of $G$ respectively $F=w^r \cdot \hat{F}(w)$ at $q' \in E=\pi^{-1}(Z)$. Then, for a point $q' \in \pi^{-1}(q)$ to be a \textnormal{kangaroo point} for $G$, i.e., fulfilling $$\textnormal{ord}_{q'}(G')=\textnormal{ord}_q(G) \textnormal{ \ and \ } \Dorder_{q'}(F') > \Dorder_q(F),$$ the following conditions must hold at $q$:
\begin{enumerate}
\item[(1)] The order $\textnormal{ord}(F)=|r|+\textnormal{ord}_q(\hat{F})$ is a multiple of $p$.
\item[(2)] The exceptional multiplicities $r_i$ at $q$ satisfy
$$\overline{r_m}+ \ldots + \overline{r_1} \leq \left( \phi_p(r)-1 \right) \cdot p,$$ where $0 \leq \overline{r_i} < p$ denote the residues of the components $r_i$ of $r=(r_m,\ldots,r_1)$ modulo $p$ and $\phi_p(r):=\#\{i \leq m; r_i \not\equiv 0 \textnormal{ mod } p\}$.
\item[(3)] The point $q'$ is determined by the expansion of $G$ at $q$. It lies on none of the strict transforms of the exceptional components $w_i=0$ for which $r_i$ is not a multiple of $p$. 
\item[(4)] The initial form of $\hat{F}$ equals, up to linear coordinate changes and multiplication by $p$-th powers, a specific homogenous polynomial, which is unique for each choice of $p$, $r$ and degree. 
\end{enumerate}
\smallskip
\end{Theo} 

\ind  The point $q$ prior to a kangaroo point will be called \textnormal{antelope point}. 
Note that for surfaces ($m=2$) condition (2) of the last theorem can be reformulated as $$r_1, r_2 \not\equiv 0 \textnormal{ mod } p \textnormal{ \ and \ } \overline{r_1}+\overline{r_2} \leq p.$$ Consequently, condition (3) implies that the point $q$ has to leave both exceptional components in order to arrive at a kangaroo point. Together this yields that an increase of the \dorder \ can only occur when applying a translational move subsequent to at least one horizontal and one vertical move.   Therefore we will analyze how the \dorder \ changes under such moves prior to the jump at the kangaroo point: 

Suppose that before this increase of the \dorder \ at the kangaroo point, already $u$ horizontal and $v$ vertical moves have taken place (in a specific order, with $u,v\geq 1$). 
Assume for sake of simplicity further that $F(y,z)$ has at the very beginning of these series of blowups been a binomial, i.e., has been of the form, $$F(y,z)=y^r z^s \cdot \left(c y^k+d z^l\right) \in K[[y,z]]/K[[y^p,z^p]]$$ with $r,s \in \N$, $k,l  \in \N_{>0}$ and $c, d \in K$. Clearly a series of $u$ horizontal and $v$ vertical moves contains at least one subsequence where a horizontal move is followed by a vertical one or the other way around. Denote by $F^{(c)}$ the transform of $F$ under the moves prior to the first of these subsequences, where $0 \leq c \leq u+v$. Note that $F^{(c)}$ is of the form $$F^{(c)}(y,z)=y^{r'}z^{s'} \cdot (c'y^{k'}+d'z^{l'})$$ with $r',s',k',l' \in \N$ and $c',d' \in K$. Since we are considering moves prior to an increase at the kangaroo point, it follows that $\Dorder(F^{(c)})=\min\left(k',l'\right)>0$. Without loss of generality assume further that afterwards first a horizontal move and then a subsequent vertical move is applied to $F^{(c)}$ (clearly the case of applying the moves in the reverse order works symmetrically). Now consider the transforms of $F^{(c)}$ under these two moves, i.e., 
$F^{(c+1)}(y,z)=F^{(c)}(yz,z)$ and $F^{(c+2)}(y,z)=F^{(c+1)}(y,yz)$ (see figure \ref{2moves}). 

	

\begin{figure}[h]
\begin{center}
\begin{minipage}{4cm}
\beginpicture

 \setcoordinatesystem units <15pt,15pt> point at 0 0

 \setplotarea x from -1 to 5, y from -1 to 5
 \axis left shiftedto x=0 ticks numbered from 6 to 5 by 5
                                short unlabeled from 1 to 5 by 1 /
 \axis bottom shiftedto y=0 ticks numbered from 6 to 5  by 5
                                short unlabeled from 1 to 5 by 1 /

 \put {\circle*{5}} [Bl] at  0  2
 \put {\circle*{5}} [Bl] at  3 0

 \put {\vector(0,1){30}} at -0.5  1
\put {\vector(0,-1){30}} at -0.5  1
\put {\small{k'}} at -1 1

\put {\vector(1,0){45}} at 1.5 -0.2
\put {\vector(-1,0){45}} at 1.5  -0.5
\put {\small{l'}} at 1.5 -1

 \plot 0 5  0 2  3 0  5 0 /

 \vshade 0 2 5 <,z,,> 3 0 5 <,z,,>  5 0 5 /


\put{$y$} at -0.8 5 \put {$z$} at 5 -0.8

 \endpicture
\hspace{1cm}
\end{minipage}
\begin{minipage}{4cm}
\beginpicture

 \setcoordinatesystem units <15pt,15pt> point at 0 0

 \setplotarea x from -1 to 5, y from -1 to 5
 \axis left shiftedto x=0 ticks numbered from 6 to 5 by 5
                                short unlabeled from 1 to 5 by 1 /
 \axis bottom shiftedto y=0 ticks numbered from 6 to 5 by 5
                                short unlabeled from 1 to 5 by 1 /
 
 \put {\circle*{5}} [Bl] at  2 2
 \put {\circle*{5}} [Bl] at  3 0

\put {\vector(0,1){30}} at 2  1
\put {\vector(0,-1){30}} at 2  1
\put {\small{k'}} at 1.5 1

\put {\vector(1,0){15}} at 2.5 -0.2
\put {\vector(-1,0){15}} at 2.5  -0.5
\put {\small{l'-k'}} at 2.5 -1

 \plot 2 5  2 2  3 0  5 0 /

 \vshade 2 2  5 <,z,,> 3 0 5 <,z,,>   5 0  5 /


\put{$y$} at -0.8 5 \put {$z$} at 5 -0.8

 \endpicture
\hspace{1cm}
\end{minipage}
\begin{minipage}{4cm}
\beginpicture

 \setcoordinatesystem units <15pt,15pt> point at 0 0

 \setplotarea x from -1 to 5, y from -1 to 5
 \axis left shiftedto x=0 ticks numbered from 6 to 5 by 5
                                short unlabeled from 1 to 5 by 1 /
 \axis bottom shiftedto y=0 ticks numbered from 6 to 5 by 5
                                short unlabeled from 1 to 5 by 1 /

 \put {\circle*{5}} [Bl] at  2 4
 \put {\circle*{5}} [Bl] at  3 3

\put {\vector(0,1){15}} at 2  3.5
\put {\vector(0,-1){15}} at 2  3.5
\put {\small{2k'-l'}} at 1.2 3.5

\put {\vector(1,0){15}} at 2.5  3.3
\put {\vector(-1,0){15}} at 2.5  3
\put {\small{l'-k'}} at 2.5 2.5

 \plot 2 5  2 4  3 3  5 3  /

 \vshade 2 4 5 <,z,,> 3 3  5 <,z,,>  5 3 5 /


\put{$y$} at -0.8 5 \put {$z$} at 5 -0.8

 \endpicture
\end{minipage}
\\[2ex]
 \begin{minipage}{12cm}
 \begin{abbildung} \label{2moves} \begin{center}
 The transforms of $F^{(c)}$ under a horizontal and a vertical move. \end{center} \end{abbildung}
 \end{minipage}
 \end{center}
\end{figure}

\noindent In the case that $N(F^{(c+2)})$ is not a quadrant, which especially presumes that $$\Dorder\left(F^{(c)}\right)=k'<l' \textnormal{ and } l'-k'<k', \hspace{20pt} (\star)$$ the \dorder \ of $F^{(c+2)}$ is given by
$\Dorder(F^{(c+2)})=\min \left(2k'-l',l'-k'\right).$
But due to $(\star)$ it follows easily (see figure \ref{halbe}) that $$\Dorder\left(F^{(c+2)}\right) \leq \frac{k'}{2} = \frac{1}{2} \cdot \Dorder\left(F^{(c)}\right).$$

\begin{figure}[h]
\begin{center}
\begin{minipage}{4cm}
\beginpicture

 \setcoordinatesystem units <15pt,15pt> point at 0 0

 \setplotarea x from 0 to 10, y from 0 to 0
 \axis bottom shiftedto y=0 ticks numbered from 0 to 0 by 1
                                short unlabeled from 0 to 10 by 1 /

\put {\circle*{5}} [Bl] at  0 0
\put {\circle*{5}} [Bl] at  5 0 
\put {\circle*{5}} [Bl] at  10 0
\put {\circle*{5}} [Bl] at  7 0 

\put{{k'}} at 5 -0.7
\put{{2k'}} at 10 -0.7
\put{{l'}} at 7 -0.7

\put {\vector(1,0){30}} at 6  0.8
\put {\vector(-1,0){30}} at 6  0.5
\put {{l'-k'}} at 6 1

\put {\vector(1,0){45}} at  8.5  0.8
\put {\vector(-1,0){45}} at 8.5  0.5
\put {{2k'-l'}} at 8.5  1

\put {\vector(1,0){75}} at 7.5  -1.0
\put {\vector(-1,0){75}} at 7.5  -1.3
\put {{k'}} at  7.9 -1.7

 \plot 0 0  10 0  /




 \endpicture
\end{minipage}
\\[2ex]
 \begin{minipage}{12cm}
 \begin{abbildung} \label{halbe} \begin{center}
 Illustration of the inequalities $(\star)$ and the value of $\Dorder(F^{(c+2)})$. \end{center} \end{abbildung}
 \end{minipage}
 \end{center}
\end{figure}

\noindent In the case that $\Dorder(F^{(c)})$ has already been smaller or equal to the half of $\Dorder(F)$, i.e., $\Dorder(F^{(c)}) \leq \frac{1}{2} \cdot \Dorder(F),$ we are  done since we have already seen that the \dorder \ can't increase under monomial moves and it thus immediately follows that
$$\Dorder\left(F^{(u+v)}\right) \leq \Dorder\left(F^{(c)}\right) \leq \frac{1}{2} \cdot \Dorder(F),$$
where $F^{(u+v)}$ denotes the transform after the $u$ horizontal and the $v$ vertical moves prior to the increase at the kangaroo point. So it remains to consider the case $\Dorder(F^{(c)}) > \frac{1}{2} \cdot \Dorder(F).$ But in this situation one has $$\Dorder\left(F^{(u+v)}\right) \leq \Dorder\left(F^{(c+2)}\right) \leq \frac{1}{2} \cdot \Dorder\left(F^{(c)}\right) \leq \frac{1}{2} \cdot \Dorder(F),$$
since clearly $\Dorder(F^{(c)}) \leq \Dorder(F)$.

\ind It is not hard to see that the previous inequalities also hold for an arbitrary series $F(y,z)$. This proves the following proposition, which is also already indicated in \cite{Hausera}:

\begin{Prop}
Let $\pi: (\widetilde{\A}^3,b') \rightarrow (\A^3,b)$ be a local point blowup with center $Z=\{b\}$ and $(x,y,z)$ local coordinates at $b$ such that $G(x,y,z)=x^p+F(y,z)$ has order $p$ at $b$. Let $b'$ be a kangaroo point for $G$ and $b$ its antelope point. Further let be given a sequence of point blowups prior to $\pi$ in a three-dimensional smooth ambient space for which the subsequent centers are equiconstant points. Let $b°$ be the last point below the antelope point $b$ where none of the exceptional components through $b$ has appeared yet. Then, the \dorder \ has dropped between $b°$ and the antelope point $b$ of the kangaroo point $b'$ at least to its half.
\end{Prop}

\ind The increase at the kangaroo point by $1$ is therefore, except in the case that the \dorder \ at the point $b°$ is equal to $1$ or $2$, in the long run dominated by the decrease of the \dorder \ in the prior blowups. By $(\star)$, one immediately sees that in the first case no increase of the \dorder \ is possible. If the \dorder \ at the point $b°$ is equal to $2$, this is not possible either. This can be checked by an easy computation using the special shape of $F$ in this case. \\



\vspace{10pt}
\bibliographystyle{plain}
\bibliography{quingsreference.bib}

\begin{thebibliography}{10}

\bibitem{Ab8}
S.~Abhyankar.
\newblock Local uniformization of algebraic surfaces over ground fields of
  characteristic {$p\ne 0$}.
\newblock {\em Ann. of Math. (2)}, 63:491--526, 1956.

\bibitem{Abhyankar-nonsplitting}
S.~Abhyankar.
\newblock Nonsplitting of valuations in extensions of two dimensional regular
  local domains.
\newblock {\em Math. Ann.}, (170):87--144, 1967.

\bibitem{MR713043}
S.~Abhyankar.
\newblock Desingularization of plane curves.
\newblock In {\em Singularities, {P}art 1 ({A}rcata, {C}alif., 1981)},
  volume~40 of {\em Proc. Sympos. Pure Math.}, pages 1--45. Amer. Math. Soc.,
  Providence, RI, 1983.

\bibitem{MR0268188}
M.~Artin.
\newblock Algebraic approximation of structures over complete local rings.
\newblock {\em Inst. Hautes \'Etudes Sci. Publ. Math.}, (36):23--58, 1969.

\bibitem{Benito_Villamayor_Surfaces}
A.~Benito and Villamayor O.
\newblock {Techniques for the study of singularities with applications to
  resolution of 2-dimensional schemes}.
\newblock {\em Math. Ann.}, 353(3):1037--1068, 2012.

\bibitem{Benito}
A.~Benito and O.~Villamayor.
\newblock Singularities in positive characteristic: elimination and and
  monoidal transformations.
\newblock arXiv:math/0811.4148, 2008.

\bibitem{MR1440306}
E.~Bierstone and P.~D. Milman.
\newblock Canonical desingularization in characteristic zero by blowing up the
  maximum strata of a local invariant.
\newblock {\em Invent. Math.}, 128(2):207--302, 1997.

\bibitem{MR2174912}
A.~Bravo, S.~Encinas, and O.~Villamayor.
\newblock A simplified proof of desingularization and applications.
\newblock {\em Rev. Mat. Iberoamericana}, 21(2):349--458, 2005.

\bibitem{Bravo_Villamayor_10}
A.~Bravo and O.~Villamayor.
\newblock {Singularities in positive characteristic, stratification and
  simplification of the singular locus}.
\newblock {\em Adv. Math.}, 224(4):1349--1418, 2010.

\bibitem{Brown}
W.~G. Brown.
\newblock On {$k$}th roots in power series rings.
\newblock {\em Math. Ann.}, 170:327--333, 1967.

\bibitem{BW}
C.~Bruschek and D.~Wagner.
\newblock \'{E}tale neighbourhoods and the normal crossings locus.
\newblock {\em Expos. Math.}, (29):133--141, 2011.

\bibitem{CJS}
V.~Cossart, U.~Jannsen, and S.~Saito.
\newblock {Canonical embedded and non-embedded resolution of singularities for
  excellent two-dimensional schemes}, 2009.
\newblock arXiv:0905.2191.

\bibitem{Cos1}
V.~Cossart and O.~Piltant.
\newblock Resolution of singularities of threefolds in positive characteristic.
  {I}. {R}eduction to local uniformization on {A}rtin-{S}chreier and purely
  inseparable coverings.
\newblock {\em J. Algebra}, 320(3):1051--1082, 2008.

\bibitem{Cos2}
V.~Cossart and O.~Piltant.
\newblock Resolution of singularities of threefolds in positive characteristic.
  {II}.
\newblock {\em J. Algebra}, 321:1836--1976, 2009.

\bibitem{Cutkosky-on-Abhyankar}
D.~Cutkosky.
\newblock A skeleton key to {A}bhyankar's proof of embedded resolution of
  characteristic $p$ surfaces.
\newblock {\em Asian J. Math.}, (15):369--416, 2012.

\bibitem{Cut1}
S.~D. Cutkosky.
\newblock {\em Resolution of singularities}, volume~63 of {\em Graduate Studies
  in Mathematics}.
\newblock American Mathematical Society, Providence, RI, 2004.

\bibitem{Cutkosky_3-Folds}
S.~D. Cutkosky.
\newblock {Resolution of singularities for 3-folds in positive characteristic}.
\newblock {\em Amer. J. Math.}, 131(1):59--127, 2009.

\bibitem{MR1949115}
S.~Encinas and H.~Hauser.
\newblock Strong resolution of singularities in characteristic zero.
\newblock {\em Comment. Math. Helv.}, 77(4):821--845, 2002.

\bibitem{EV}
S.~Encinas and O.~Villamayor.
\newblock Rees algebras and resolution of singularities.
\newblock {\em Rev. Mat. Iberoamericana. Proceedings XVI-Coloquio
  Latinoamericano de Algebra 2006.}

\bibitem{FaberHauser}
E.~Faber and H.~Hauser.
\newblock {T}oday's {M}enu: {G}eometry and resolution of singular algebraic
  surfaces.
\newblock {\em Bull. Amer. Math. Soc.}, 47(3):373--417, 2010.

\bibitem{Hausera}
H.~Hauser.
\newblock Wild singularities and kangaroo points for the resolution in positive
  characteristic.
\newblock Preprint 2010, available on www.hh.hauser.cc.

\bibitem{MR1748627}
H.~Hauser.
\newblock Excellent surfaces and their taut resolution.
\newblock In {\em Resolution of singularities (Obergurgl, 1997)}, volume 181 of
  {\em Progr. Math.}, pages 341--373. Birkh\"auser, Basel, 2000.

\bibitem{MR2063657}
H.~Hauser.
\newblock Three power series techniques.
\newblock {\em Proc. London Math. Soc. (3)}, 89(1):1--24, 2004.

\bibitem{HauserBull}
H.~Hauser.
\newblock On the problem of resolution of singularities in positive
  characteristic (or: A proof that we are still waiting for).
\newblock {\em Bull. Amer. Math. Soc.}, 47(1):1--30, 2010.

\bibitem{HauSchi}
H.~Hauser and J.~Schicho.
\newblock A game for the resolution of singularities.
\newblock {\em Proc. London Math. Soc. (3)}, 105(1):1149--1182, 2012.

\bibitem{Hironaka}
H.~Hironaka.
\newblock Desingularization of excellent surfaces.
\newblock Notes by B. Bennett at the Conference on Algebraic Geometry, Bowdoin
  1967. Reprinted in: Cossart, V., Giraud, J., Orbanz, U.: Resolution of
  surface singularities. Lecture Notes in Mathematics 1101, Springer 1984.

\bibitem{HironakaTrieste}
H.~Hironaka.
\newblock A program for resolution of singularities, in all characteristics
  $p>0$ and in all dimensions.
\newblock Lecture notes from the summer school and conference on resolution of
  singularities, ICTP Trieste, June 2006.

\bibitem{HironakaRIMS}
H.~Hironaka.
\newblock Program for resolution of singularities in characteristic $p>0$.
  {N}otes from lectures at {R}{I}{M}{S} {K}yoto, {D}ecember 2008.

\bibitem{HironakaClay}
H.~Hironaka.
\newblock Program for resolution of singularities in characteristic $p>0$.
  {N}otes from lectures at the {C}lay {M}athematics {I}nstitute, {S}eptember
  2008.

\bibitem{HironakaTordesillas}
H.~Hironaka.
\newblock Resolution of singularities. {M}anuscript 138 pp, {T}ordesillas,
  {S}eptember 2011.

\bibitem{MR0199184}
H.~Hironaka.
\newblock Resolution of singularities of an algebraic variety over a field of
  characteristic zero. {I}, {II}.
\newblock {\em Ann. of Math. (2) 79 (1964), 109--203; ibid. (2)}, 79:205--326,
  1964.

\bibitem{MR1996845}
H.~Hironaka.
\newblock Theory of infinitely near singular points.
\newblock {\em J. Korean Math. Soc.}, 40(5):901--920, 2003.

\bibitem{Hiraku1}
H.~Kawanoue.
\newblock Toward resolution of singularities over a field of positive
  characteristic. {I}. {F}oundation; the language of the idealistic filtration.
\newblock {\em Publ. Res. Inst. Math. Sci.}, 43(3):819--909, 2007.

\bibitem{Kawanoue_Matsuki}
H.~Kawanoue and K.~Matsuki.
\newblock {Toward resolution of singularities over a field of positive
  characteristic (the idealistic filtration program) {P}art {II}. {B}asic
  invariants associated to the idealistic filtration and their properties}.
\newblock {\em Publ. Res. Inst. Math. Sci.}, 46(2):359--422, 2010.

\bibitem{Kawanoue_Matsuki_Surfaces}
H.~Kawanoue and K.~Matsuki.
\newblock {Resolution of singularities of an idealistic filtration in dimension
  3 after Benito-Villamayor}.
\newblock arXiv:1205.4556, 2012.

\bibitem{MR2289519}
J.~Koll{\'a}r.
\newblock {\em Lectures on resolution of singularities}, volume 166 of {\em
  Annals of Mathematics Studies}.
\newblock Princeton University Press, Princeton, 2007.

\bibitem{Lipman}
J.~Lipman.
\newblock Desingularization of two-dimensional schemes.
\newblock {\em Ann. Math. (2)}, 107(1):151--207, 1978.

\bibitem{MR935710}
T.~T. Moh.
\newblock On a stability theorem for local uniformization in characteristic
  {$p$}.
\newblock {\em Publ. Res. Inst. Math. Sci.}, 23(6):965--973, 1987.

\bibitem{MR1395176}
T.~T. Moh.
\newblock On a {N}ewton polygon approach to the uniformization of singularities
  of characteristic {$p$}.
\newblock In {\em Algebraic geometry and singularities (La R\'abida, 1991)},
  volume 134 of {\em Progr. Math.}, pages 49--93. Birkh\"auser, Basel, 1996.

\bibitem{MR715854}
S.~B. Mulay.
\newblock Equimultiplicity and hyperplanarity.
\newblock {\em Proc. Amer. Math. Soc.}, 89(3):407--413, 1983.

\bibitem{MR684627}
R.~Narasimhan.
\newblock Hyperplanarity of the equimultiple locus.
\newblock {\em Proc. Amer. Math. Soc.}, 87(3):403--408, 1983.

\bibitem{MR715853}
R.~Narasimhan.
\newblock Monomial equimultiple curves in positive characteristic.
\newblock {\em Proc. Amer. Math. Soc.}, 89(3):402--406, 1983.

\bibitem{MR985852}
O.~Villamayor.
\newblock Constructiveness of {H}ironaka's resolution.
\newblock {\em Ann. Sci. \'Ecole Norm. Sup. (4)}, 22(1):1--32, 1989.

\bibitem{MR1198092}
O.~Villamayor.
\newblock Patching local uniformizations.
\newblock {\em Ann. Sci. \'Ecole Norm. Sup. (4)}, 25(6):629--677, 1992.

\bibitem{MR2163383}
J.~W{\l}odarczyk.
\newblock Simple {H}ironaka resolution in characteristic zero.
\newblock {\em J. Amer. Math. Soc.}, 18(4):779--822, 2005.

\bibitem{Zeillinger}
D.~Zeillinger.
\newblock A short solution to {H}ironaka's polyhedra game.
\newblock {\em L'Enseign. Math.}, 52:143--158, 2006.

\end{thebibliography}


\vspace{25pt}

\noindent \ind Faculty of Mathematics\\ 
University of Vienna, and \\
Institut f\"ur Mathematik\\ University of Innsbruck, Austria \\
\ind herwig.hauser@univie.ac.at \\
dominique.wagner@univie.ac.at

\end{document}